\documentclass[10pt]{article}

\usepackage{setspace}
\usepackage{placeins}
\usepackage{microtype}
\usepackage{multicol}

\usepackage{url}
\usepackage{hyperref}

\usepackage[numbers]{natbib}
\let\citep\cite
\let\citet\cite

\usepackage{enumitem}
\usepackage{times}

\usepackage{amsmath}
\usepackage{amsthm}
\usepackage{amssymb}
\usepackage{amsopn}

\usepackage{bbm}
\usepackage{bm}
\usepackage{nicefrac}

\usepackage{mathtools}
\usepackage{tikz}

\usepackage{graphicx}
\usepackage[export]{adjustbox}

\usepackage{float}
\usepackage{wrapfig}
\usepackage{subcaption}

\usepackage{comment}
\usepackage{xcolor}
 
\usepackage{algorithm}
\usepackage{algorithmic}
\usepackage{listings}

\newtheorem{theorem}{Theorem}
\newtheorem{corollary}{Corollary}
\newtheorem{proposition}{Proposition}
\newtheorem{lemma}{Lemma}

\newtheorem{remark}{Remark}

\DeclareMathOperator*{\argmin}{\arg\!\min}

\DeclareMathOperator{\Nn}{\mathcal{N}}

\DeclareMathOperator{\Ii}{I}

\DeclareMathOperator{\EE}{\mathbb{E}}
\DeclareMathOperator{\RR}{\mathbb{R}}

\DeclareMathOperator{\Prob}{\mathbb{P}}

\DeclareMathOperator{\bb}{\mathbf{b}}
\DeclareMathOperator{\cc}{\mathbf{c}}
\DeclareMathOperator{\ee}{\mathbf{e}}
\DeclareMathOperator{\hh}{\mathbf{h}}
\DeclareMathOperator{\qq}{\mathbf{q}}
\DeclareMathOperator{\sss}{\mathbf{s}}
\DeclareMathOperator{\uu}{\mathbf{u}}
\DeclareMathOperator{\vv}{\mathbf{v}}
\DeclareMathOperator{\ww}{\mathbf{w}}
\DeclareMathOperator{\xx}{\mathbf{x}}
\DeclareMathOperator{\yy}{\mathbf{y}}

\DeclareMathOperator{\bzr}{\mathbf{0}}
\DeclareMathOperator{\bones}{\mathbf{1}}
\DeclareMathOperator{\beps}{\bm{\epsilon}}
\DeclareMathOperator{\bmu}{\bm{\mu}}

\usetikzlibrary{shapes, arrows, positioning}

\tikzstyle{decision} = [diamond, draw, fill = green!20, 
    text width = 4.5em, text badly centered, node distance = 3cm, inner sep = 0pt]
\tikzstyle{block} = [rectangle, draw, fill = blue!20, 
    text width = 10em, text centered, rounded corners, minimum height = 4em]
\tikzstyle{line} = [draw, -latex']
\tikzstyle{cloud} = [draw, ellipse, fill = red!20, node distance = 3cm,
    minimum height = 2em]
\tikzstyle{circ} = [circle, draw, fill = red!20, 
    text width = 10em, text centered, minimum height = 2em]
    
\begin{document}

\title{Time Series Source Separation using Dynamic Mode Decomposition\footnote{This work extends the work in \citet{prasadanDMDconf}. Submitted to the Editors on 2019 July. To appear in SIADS. This work was supported by ONR grant N00014-15-1-2141, DARPA Young Faculty Award D14AP00086, ARO MURI W911NF-11-1-039.}}
% Authors: full names plus addresses.%
\author{Arvind Prasadan and Raj Rao Nadakuditi\footnote{University of Michigan, Ann Arbor, MI 
  (prasadan{@}umich.edu, rajnrao{@}umich.edu)}}

\maketitle

\begin{abstract}%196 words
The Dynamic Mode Decomposition (DMD) extracted dynamic modes are the non-orthogonal eigenvectors of the matrix that best approximates  the one-step temporal evolution of the multivariate samples. In the context of dynamical system analysis, the extracted dynamic modes  are a generalization of global stability modes. We apply DMD to a data matrix whose rows are linearly independent, additive mixtures of latent time series. We show that when the latent  time series are uncorrelated at a lag of one time-step then, in the large sample limit, the recovered dynamic modes will approximate, up to a column-wise normalization, the columns of the mixing matrix. Thus, DMD is a time series blind source separation algorithm in disguise, {but is different from closely related second order algorithms such as the Second-Order Blind Identification (SOBI) method and the Algorithm for Multiple Unknown Signals Extraction (AMUSE). All can unmix mixed stationary, ergodic Gaussian time series in a way that kurtosis-based Independent Components Analysis (ICA)} fundamentally cannot. We use our insights on single lag DMD to develop a higher-lag extension, analyze the finite sample performance with and without randomly missing data, and identify settings where the higher lag variant can outperform the conventional single lag variant. We validate our results with numerical simulations, and highlight how DMD can be used in change point detection. 
\end{abstract}

% 37M10: Dynamical Systems: Time Series %
% 62M10: Time series %
% 47A55: Linear Operator Perturbations %
% 47A75: Linear Eigenvalue Problems %
% 60G35: Signal Detection and Filtering %
% 37N30: Dynamical Systems in Numerical Analysis %

% Section: Introduction %
\section{Introduction}

The Dynamic Mode Decomposition (DMD) algorithm was invented by P. Schmid as a method for extracting dynamic information from temporal measurements of a multivariate fluid flow vector \citep{schmid2010dynamic}. The dynamic modes  extracted are the generically non-orthogonal eigenvectors of a non-normal matrix that best linearizes the one-step evolution of the measured vector (to be quantified in what follows). 

Schmid showed that the dynamic modes recovered by DMD correspond to the globally stable modes in the flow \citep{schmid2010dynamic}. The non-orthogonality of the recovered dynamic modes  reveals spatial structure in the temporal evolution of the measured fluid flows in a way that other second order spatial correlation  based methods, such as the Proper Orthogonal Decomposition (POD), do not \citep{kerschen2005method}. This spurred follow-on work on other applications and extensions of DMD to understanding dynamical systems from measurements. 

\subsection{Previous work on DMD and the analysis of dynamical systems}

Early analyses of the DMD algorithm drew connections between the DMD modes and the eigenfunctions of the Koopman operator from dynamical system theory. Rowley et al. and Mezi{\'c} et al. showed that under certain conditions, the DMD modes approximate the eigenfunctions of the Koopman operator for a given system \citep{rowley2009spectral, mezic2013analysis}. Related work in \citet{bagheri2013koopman} studied the Koopman operator directly, analyzed its spectrum, and compared it against the spectrum of the matrix decomposed in DMD. The work in \citet{rowley2009spectral} also explained how the linear DMD modes can elucidate the  structure in the temporal evolution in  nonlinear fluid flows. The work in \citet{vcrnjaric2017koopman} provided a further analysis of the Koopman operator and more connections to DMD. More recently, Lusch et al. have shown how  deep learning can be combined with DMD to extract modes for a non-linearly evolving dynamical system \citep{lusch2018deep}. 

There have been several extensions of DMD. The authors in \citet{chen2012variants} developed a method to improve the robustness of DMD to noise. Jovanovic et al. proposed a sparsity-inducing formulation of DMD that allowed fewer dynamic modes  to better capture the dynamical system \citep{jovanovic2014sparsity}. Tu et al. developed a DMD variant that takes into account systematic measurement errors and measurement noise \citep{tu2014dynamic}; this framework was extended in \citet{hemati2017biasing}. A Bayesian, probabilistic variant of DMD was developed in \citet{ijcai2017-392}, where a Gibbs sampler for the modes and a sparsity-inducing prior were proposed. Another recent extension of DMD includes an online (or streaming) version of DMD  \citep{zhang2017online}.

Additionally, there have been applications of DMD to other domains besides computational fluid mechanics. The work in \citet{bai2017dynamic} applied DMD to compressed sensing settings. A related work applied DMD to model the background in a streaming video \citep{kutz2017dynamic}. The authors in \citet{mann2016dynamic} applied DMD to finance, by using the predicted modes and temporal variations to forecast future market trends. The authors in \citet{berger2015estimation} brought DMD to the field of robotics, and used DMD to estimate perturbations in the motion of a robot. DMD has also been applied to power systems analysis, where it has been used to analyze transients in large power grids \citep{barocio2015dynamic}. There are many more applications and extensions, and we point the interested reader to the recent book by Kutz et al. \citep{kutz2016book}.

\subsection{Our main finding: DMD unmixes lag-1 (or higher lag) uncorrelated time series} 

We will introduce the general problem and model in Section \ref{sec:model}, but before proceeding, we will consider a simple, illustrative example. Suppose that we are given multivariate observations $\xx_t \in \RR^{p}$ modeled as  
\begin{equation}\label{eq:mixing model}
\xx_t = H \sss_t = Q D \sss_t,
\end{equation}
where $t$ is an integer, $H = Q D \in \RR^{p \times p}$ is a non-singular mixing matrix, and $\sss_t \in \RR^{p}$ is the latent vector of {random} signals (or sources).  The matrix $Q \in \RR^{p \times p}$ has unit-norm columns and is related to $H$ by
\begin{equation}
Q = \begin{bmatrix} \qq_1 & \ldots & \qq_p \end{bmatrix} = \begin{bmatrix} \dfrac{\hh_1}{\left\|\hh_1\right\|_2} & \ldots & \dfrac{\hh_p}{\left\|\hh_p\right\|_2} \end{bmatrix}.
\end{equation}
Setting entries of the diagonal matrix $D = \textrm{diag}(d_1, \ldots, d_p)$ as $d_i = \left\|\hh_i\right\|_2$ ensures that $H = Q D$ as in (\ref{eq:mixing model}). Note that by the phrase `mixing matrix', we mean that $H \sss_t$ produces a linear combination of the coordinates of $\sss_t$, i.e., a mixing of the coordinates. 

In what follows, we will adopt the following notational convention: we shall use boldface to denote vectors such as $\sss_t$. Matrices, such as $H$, will be denoted by non-boldface upper-case letters; and scalars, such as $s_{t1}$, will be denoted by lower-case symbols.  

We assume, without loss of generality, that
\begin{equation}
 \EE \left[\sss_t\right] = \bzr_p \textrm{ and } \EE \left[\sss_t \sss_t^T\right] = \Ii_p.  
\end{equation}
The lag-$\tau$ covariance matrix of $\sss_t$ is defined as
\begin{equation}\label{eq:Ltau1}
\EE [L_{\tau}] = \EE \left[\sss_t \sss_{t + \tau}^T\right] = \EE \left[ \sss_{t + \tau} \sss_t^T\right], 
\end{equation}
where $\tau$ is a  non-negative integer. 

If we are able to form a reliable estimate $\widehat{H}$ of the mixing matrix $H$ from the $n$ multivariate observations $\xx_1, \ldots, \xx_n$ then, via Eq. (\ref{eq:mixing model}), we can unmix the latent signals $\sss_t$ by computing $\widehat{H}^{-1} \xx_t$. Inferring $Q$ and computing $\widehat{Q}^{-1} \xx_t$  will also similarly unmix the signals. Inferring the mixing matrix and unmixing the signals (or sources)  is referred to as \textit{blind source separation} \citep{choi2005blind}.

Our key finding is that when the lag-1 covariance matrix $\EE[L_{1}]$ in (\ref{eq:Ltau1}) is diagonal, corresponding to the setting where the latent signals are lag-1 uncorrelated, weakly stationary time series, and there are sufficiently many samples of $\xx_t$, then the DMD algorithm in (\ref{eq:dmd opt}) produces a non-normal matrix whose non-orthogonal eigenvectors are reliably good (to be quantified in what follows) estimates of $Q$ in (\ref{eq:mixing model}).  In other words, DMD unmixes lag-1 uncorrelated signals and weakly stationary time series. 

Our findings reveal that a straightforward  extension of DMD, described in Section \ref{sec:lagdmd} and (\ref{eq:dmdtauopt}), allows $\tau$-DMD to unmix lag $\tau$ uncorrelated signals and time series. This brings up the possibility of using a higher lag $\tau$ to unmix signals that might exhibit a more favorable correlation at larger lag $\tau$ than at a lag of one. Indeed, in Figure \ref{fig:arma} we provide one such example where $2$-DMD provides a better estimate of $Q$ than does $1$-DMD.

Our main contribution, which builds on our previous work in \citet{prasadanDMDconf}, is the analysis of the unmixing performance of DMD and  $\tau$-DMD (introduced in Section \ref{sec:lagdmd}), when unmixing deterministic signals and random, weakly stationary time series in the finite sample regime and in the setting where there is randomly missing data in the observations $\xx_t$. 

\subsection{New insight:  DMD can unmix ergodic time series that kurtosis-based ICA cannot}  

Independent Component Analysis (ICA) is  a classical algorithm for blind source separation \citep{lee1998independent,mitsui2017blind} that is often used for the cocktail party problem of unmixing mixed audio signals. Our analysis reveals that DMD can be succesfully applied to this problem as well because independent audio sources are well modeled as one-lag (or higher lag) uncorrelated (see Figure \ref{fig:audio}). 

{It is known that kurtosis- or cumulant-based ICA (hereafter refered to as ICA) fails when more then one of the independent, latent signals is normally distributed \cite[Ch.~7]{hyvarinen2001independent}. A consequence of this is that ICA will fail to unmix mixed independent, ergodic time series with Gaussian marginal distributions: each latent signal will have a kurtosis of zero.} Our analysis, {culminating in Theorem \ref{thm:armacorrtau},} reveals that DMD will succeed in this setting, even as ICA fails; see Figure \ref{fig:ar1_intro} for an illustration where ICA fails to unmix two mixed, independent Gaussian AR(1) processes while DMD succeeds. Note that these are two independent realizations of AR(1) processes, and that there is no averaging over several realizations. Thus, DMD can and should be used by practitioners to re-analyze multivariate time series data for which the use of ICA has not revealed any insights.  

\begin{figure*}[!htb]
\begin{multicols}{4}
\begin{minipage}[b]{\linewidth}
\centering
\centerline{\includegraphics[width = \textwidth, trim = {0 0 0 0.5cm}, clip]{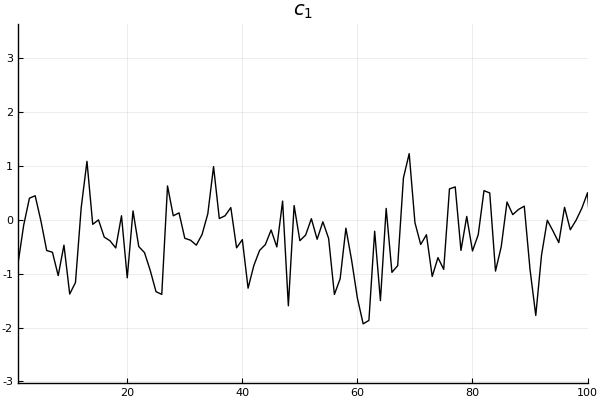}}
\subcaption{$AR(1)$, $0.7$}
\end{minipage}
\begin{minipage}[b]{\linewidth}
\centering
\centerline{\includegraphics[width = \textwidth, trim = {0 0 0 0.5cm}, clip]{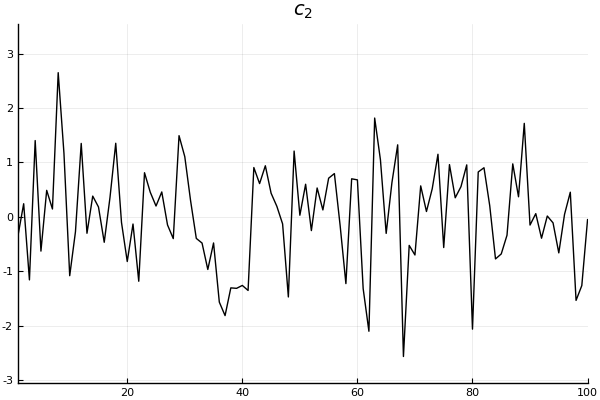}}
\subcaption{$AR(1)$, $0.2$}
\end{minipage}
\begin{minipage}[b]{\linewidth}
\centering
\centerline{\includegraphics[width = \textwidth, trim = {0 0 0 0.5cm}, clip]{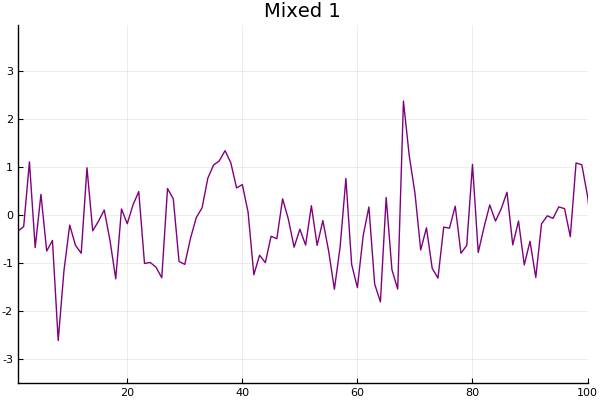}}
\subcaption{Mixed 1}
\end{minipage}
\begin{minipage}[b]{\linewidth}
\centering
\centerline{\includegraphics[width = \textwidth, trim = {0 0 0 0.5cm}, clip]{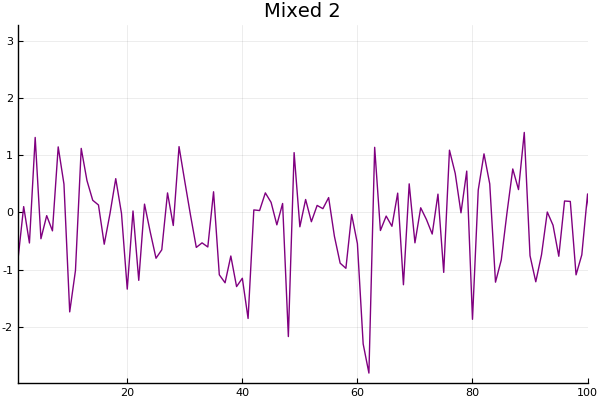}}
\subcaption{Mixed 2}
\end{minipage}
\begin{minipage}[b]{\linewidth}
\centering
\centerline{\includegraphics[width = \textwidth, trim = {0 0 0 0.5cm}, clip]{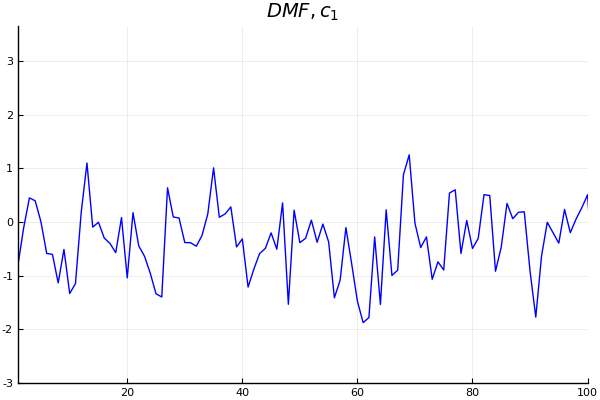}}
\subcaption{DMD 1}
\end{minipage}
\begin{minipage}[b]{\linewidth}
\centering
\centerline{\includegraphics[width = \textwidth, trim = {0 0 0 0.5cm}, clip]{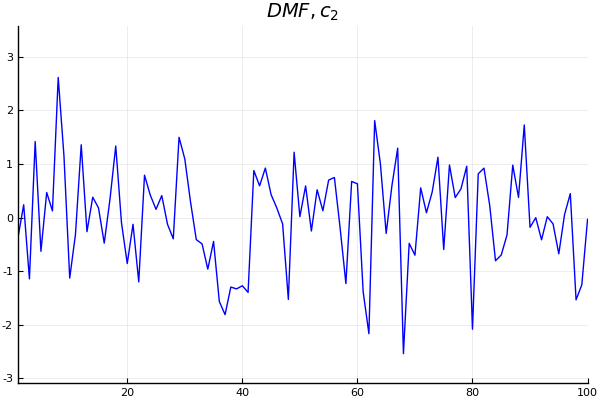}}
\subcaption{DMD 2}
\end{minipage}
\begin{minipage}[b]{\linewidth}
\centering
\centerline{\includegraphics[width = \textwidth, trim = {0 0 0 0.5cm}, clip]{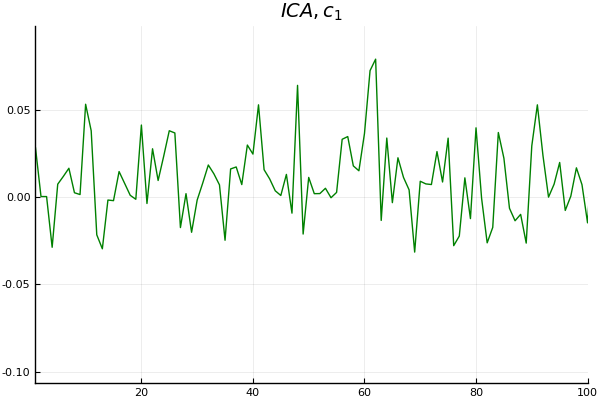}}
\subcaption{ICA 1}
\end{minipage}
\begin{minipage}[b]{\linewidth}
\centering
\centerline{\includegraphics[width = \textwidth, trim = {0 0 0 0.5cm}, clip]{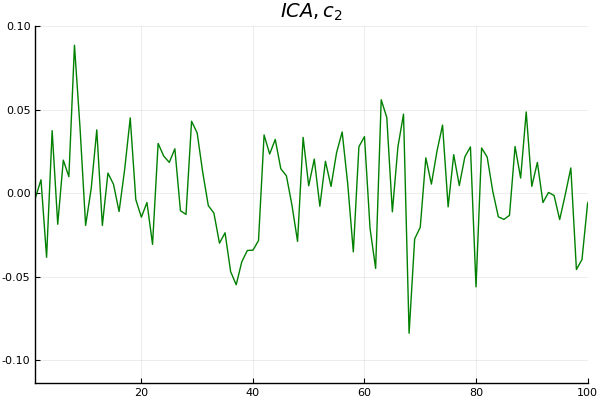}}
\subcaption{ICA 2}
\end{minipage}
\end{multicols}
\vspace{-0.7cm}
\caption{We generate two AR(1) signals of length $n = 1000$, with coefficients $0.2$ and $0.7$ respectively. We mix them orthogonally, and compare the performance of ICA and DMD at unmixing them. We observe that the squared error, defined in (\ref{eq:S_err}), of ICA is $0.41$, whereas that from DMD is $0.0055$. Indeed, ICA fails because the marginal distribution of each AR(1) process is Gaussian. In these plots, for ease of visualization we plot the first $100$ samples. }
\label{fig:ar1_intro}
\vspace{-0.4cm}
\end{figure*}

\subsection{New insight: DMD can unmix mixed Fourier series that PCA cannot}

{Principal Component Analysis (PCA) is a standard, linear dimensionality reduction method \cite{johnstone2009consistency} that can be expressed in terms of the singular value decomposition (SVD) of a data matrix.} The eigenwalker model, described in \citet{troje2002little}, is a linear model for human motion. The model is a linear combination of vectors, via
\begin{equation} \label{eq:eigwalk}
    \xx_t = \sum_{i = 1}^k \qq_i \cos\left(\omega_i t + \phi_i\right).
\end{equation}
The vectors $\qq_i$ are the modes of the motion, and each has a sinusoidal temporal variation. We generate our model as follows: 
$$\xx_t = \qq_1 \cos\left(2 t\right) + \qq_2 \cos\left(t / 4\right),$$
for $t = 1$ to $1000$, where 
$Q = \begin{bmatrix} \qq_1 & \qq_2\end{bmatrix} = \begin{bmatrix} 1/3 & 2 / \sqrt{5} \\ 2/3 & 1 / \sqrt{5} \\ 2/3 & 0 \end{bmatrix}$.
This model has been decomposed with ICA, and used for video motion editing and analysis \citep{shapiro2006style}. Here, we apply PCA and compare it to DMD. In Figure \ref{fig:eigwalk_intro}, we display the results of unmixing with PCA and with DMD. We observe that DMD successfully unmixes the cosines, while PCA fails: note that unless the $\qq_i$ are orthogonal, there is no hope of a successful unmixing. Moreover, the estimation of of $Q$ from PCA fails, as we find that
$\widehat{Q}_{PCA} = \begin{bmatrix} -0.686895 &  0.624695\\-0.623497 & -0.243983\\ -0.373399 &  -0.741774\end{bmatrix}$,
 which has a squared error of $0.81$, while the estimate from DMD has a squared error of $2.9 \times 10^{-7}$, where the error is computed according to (\ref{eq:generalboundshift}).

\begin{figure*}[htb]
\centering
\begin{multicols}{4}
\begin{minipage}[b]{0.9\linewidth}
\centering
\centerline{\includegraphics[width = \textwidth, trim = {0 0 0 0.5cm}, clip]{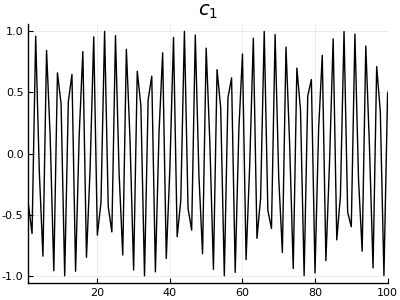}}
\subcaption{$\cos\left(2 t\right)$}
\end{minipage}
\begin{minipage}[b]{0.9\linewidth}
\centering
\centerline{\includegraphics[width = \textwidth, trim = {0 0 0 0.5cm}, clip]{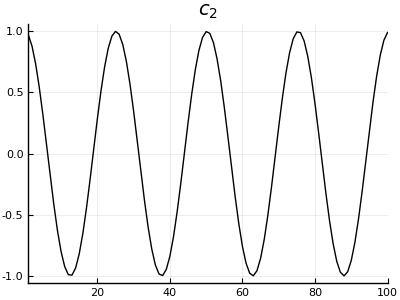}}
\subcaption{$\cos\left(t /4\right)$}
\end{minipage}
\begin{minipage}[b]{0.9\linewidth}
\centering
\centerline{\includegraphics[width = \textwidth, trim = {0 0 0 0.5cm}, clip]{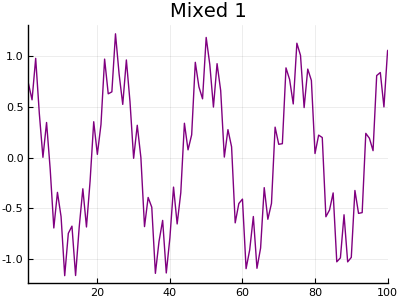}}
\subcaption{Mixed 1}
\end{minipage}
\begin{minipage}[b]{0.9\linewidth}
\centering
\centerline{\includegraphics[width = \textwidth, trim = {0 0 0 0.5cm}, clip]{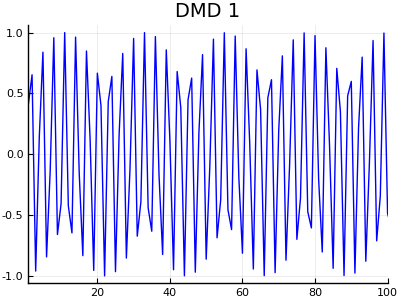}}
\subcaption{DMD 1}
\end{minipage}
\begin{minipage}[b]{0.9\linewidth}
\centering
\centerline{\includegraphics[width = \textwidth, trim = {0 0 0 0.5cm}, clip]{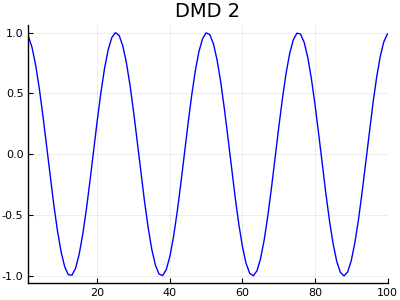}}
\subcaption{DMD 2}
\end{minipage}
\begin{minipage}[b]{0.9\linewidth}
\centering
\centerline{\includegraphics[width = \textwidth, trim = {0 0 0 0.5cm}, clip]{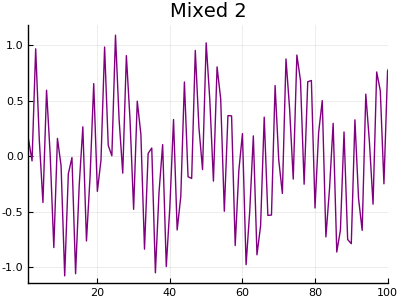}}
\subcaption{Mixed 2}
\end{minipage}
\begin{minipage}[b]{0.9\linewidth}
\centering
\centerline{\includegraphics[width = \textwidth, trim = {0 0 0 0.5cm}, clip]{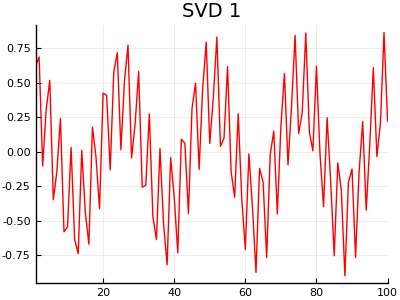}}
\subcaption{PCA 1}
\end{minipage}
\begin{minipage}[b]{0.9\linewidth}
\centering
\centerline{\includegraphics[width = \textwidth, trim = {0 0 0 0.5cm}, clip]{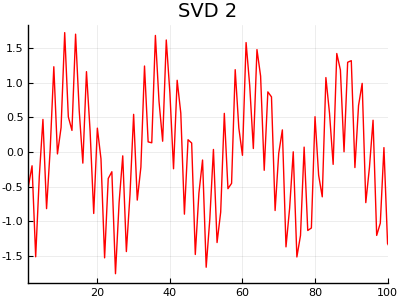}}
\subcaption{PCA 2}
\end{minipage}
\begin{minipage}[b]{0.9\linewidth}
\centering
\centerline{\includegraphics[width = \textwidth, trim = {0 0 0 0.5cm}, clip]{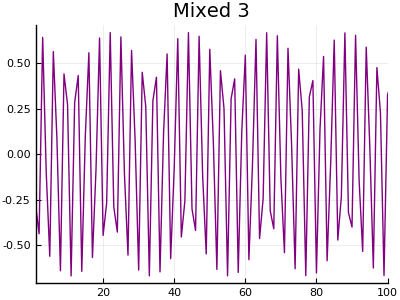}}
\subcaption{Mixed 3}
\end{minipage}
\end{multicols}
\vspace{-0.7cm}
\caption{We generate data according to the eigenwalker model (\ref{eq:eigwalk}), and use DMD and PCA to recover the cosine signals. We observe that DMD recovers the signals, while PCA does not. Indeed, we observe that the squared error for the recovered cosines, defined in (\ref{eq:S_err}), from PCA is $1.97$, whereas that from DMD is $4.57 \times 10^{-7}$. For ease of visualization, we zoom in on the first $100$ samples. }
\label{fig:eigwalk_intro}
\vspace{-0.4cm}
\end{figure*}

\subsection{Connection with other algorithms for time series blind source separation}

Let $H = U \Sigma V^T$ be the singular value decomposition (SVD) of $H$. Then, we have that $\EE[\xx_t] = \bzr_p$ and
\begin{equation}\label{eq:cov}
   \Sigma_{\xx \xx} = \EE [\xx_t \xx_t^T] = H H^T  = U \Sigma^2 U^T.
\end{equation}
Given $\Sigma_{xx}$ and $\xx_t$, we can compute the whitened vector 
\begin{equation}\label{eq:wt}
    \ww_t =  \Sigma_{\xx \xx}^{-1/2} \xx_t,
\end{equation}
whose covariance matrix is given by $\EE [\ww_t \ww_t^T ] = \Ii_p$. Then from (\ref{eq:mixing model}) and (\ref{eq:cov}) we have that
\begin{equation}\label{eq:wt2 model}
    \ww_t =  (U V^T) \sss_t,
\end{equation}
where the mixing matrix $UV^T$ is an orthogonal matrix because the $U$ and $V$ matrices, which correspond to the left and right singular vector matrices of $H$ in (\ref{eq:mixing model}) are orthogonal. 

Equation (\ref{eq:wt2 model}) reveals that we can solve the blind source separation problem and unmix $\sss_t$ from observations of $\ww_t$  if we can infer the orthogonal mixing matrix $UV^T$ from data. To that end, we note that
\begin{equation}\label{eq:whitened lagcorr}
\EE \left[\ww_t \ww_{t+ \tau}^T\right] =  (U V^T) \EE \left[\sss_t \sss_{t+ \tau}^T\right] (UV^T)^T = (UV^T) \EE [L_{\tau}] (UV^T)^T.
\end{equation}

Equation (\ref{eq:whitened lagcorr}) reveals that when the latent signals $\sss_t$ are lag-1 uncorrelated, \textit{i.e.}, $\EE [L_{1}]$ is a diagonal matrix, then the lag-1 covariance matrix of the whitened vector $\ww_t$ will be diagonalized by the orthogonal matrix $UV^T$. The sample lag-1 covariance matrix computed from finite data will, in general, not be symmetric and so we might infer $UV^T$ from the eigenvectors of the symmetric part: this leads to the  AMUSE (Algorithm for Multiple Unknown Signals Extraction) method \citep{tong1990amuse}.

A deeper inspection of (\ref{eq:whitened lagcorr}) reveals that if $\sss_t$ are second order, weakly stationary time series that are uncorrelated for multiple values of $\tau$ (corresponding to multiple lags), then we can infer $(UV^T)$ (which, incidentally corresponds to the  polar part of the polar decomposition of the mixing matrix $H$ in (\ref{eq:mixing model})) by posing it as  joint-diagonalization of $\EE \left[\ww_t \ww_{t+ \tau_i}^T\right] $ for $l$ lags corresponding to $\tau_1, \ldots, \tau_l$. This is the basis of the Second-Order Blind Identification (SOBI) method \citep{belouchrani1997blind} where the joint diagonalization problem is addressed by finding the orthogonal matrix $\Gamma$ that minimizes the sums-of-squares of the off-diagonal entries of $\Gamma^T \EE \left[\ww_t \ww_{t+ \tau_i}^T\right] \Gamma$. Numerically, this problem is solvable via the JADE method \citep{cardoso1993blind, JADEPackage, miettinen2016separation}.

Miettinen et al analyze the performance of a symmetric variant of the SOBI method in \citet{miettinen2016separation} and the problem of determining the number of latent signals that are distinct from white noise in \citet{matilainen2018number}. Their results for the performance are asymptotic and distributional. That is, the limiting distribution of the estimated matrix $\Gamma$ is computed, when the input signals are realizations of some time series, with zero mean and diagonal autocorrelations at every lag $\tau \in \{0, 1, 2, \ldots\}$. As will be seen in what follows, these assumptions are very similar to those that we impose on DMD. Our analysis for the missing data setting is new and has no counter-part in the SOBI or AMUSE performance analysis literature. 

In Table \ref{tab:algorithms}, we summarize  the various algorithms for unmixing of stationary time series. Table \ref{tab:algorithms} brings into sharp focus the manner in which DMD and $\tau$-DMD are similar to and different from the AMUSE and SOBI algorithms. All algorithms diagonalize a matrix; SOBI and AMUSE estimate orthogonal matrices while DMD and $\tau$-DMD estimate non-orthogonal matrices. The SOBI and AMUSE algorithms diagonalize cross-covariance matrices formed from whitened time series data  while DMD and $\tau$-DMD works on the time series data directly. Thus SOBI and AMUSE explicitly whiten the data while DMD implicitly whitens the data. SOBI and DMD exhibit similar performance (see Figure. \ref{fig:SOBI}) -- a more detailed theoretical study comparing their performance in the noisy setting is warranted.

\begin{table*}[htb]
\small
    \centering
    {\renewcommand{\arraystretch}{2.1}%
    \begin{tabular}{|c|c|c|c|}
    \hline
    Algorithm   &   Key Matrix  &   Fit for Key Matrix   &   Numerical Method    \\
    \hline
    DMD         &   $\widehat{A} = X_{(1)} \left[X_{(0)}\right]^+$     &   $Q L_1 Q^{+}$, $Q$ non-orthogonal   &   Non-Symmetric Eig.   \\
    \hline
    $\tau$-DMD  &   $\widehat{A}_{\tau} = X_{(1)}^{\tau} \left[X_{(0)}^{\tau} \right]^{+}$   & $Q L_{\tau} Q^{+}$, $Q$ non-orthogonal   &   Non-symmetric Eig.    \\
    \hline
    AMUSE   &  $\widehat{A}_{\tau} = Y_{(1)}^{\tau}  \left[Y_{(0)}^{\tau} \right]^{T}$   &   $\Gamma L_{\tau} \Gamma^T$, $\Gamma$ orthogonal   &  Eig. of Symmetric part     \\
    \hline
    SOBI    &   \shortstack[c]{$\widehat{A}_{\tau_i} = Y_{(1)}^{\tau_i}  \left[Y_{(0)}^{\tau_i} \right]^{T}$,\\$i \in \{1, 2, \ldots l\}$}    &     $\Gamma L_{\tau_i} \Gamma^T$, $\Gamma$ orthogonal     &   Joint Diagonalization \\
    \hline
    \end{tabular}}
    \caption{Comparison of the various second order algorithms for time series blind source separation. Here $Y = \left[X X^T\right]^{-1/2} X$, is the whitened data matrix and $Y_{(0)}^{\tau}$ and $Y_{(1)}^{\tau}$ are defined analogous to $X_{(0)}^{\tau}$ and $X_{(1)}^{\tau}$, as in (\ref{eq:X0X1}), (\ref{eq:X0X1tau}), and (\ref{eq:Ltau}), respectively.}
    \label{tab:algorithms}
\end{table*}

\subsection{Organization}

The remainder of this paper is organized as follows. In Section \ref{sec:model}, we introduce the time series data matrix model and describe the DMD algorithm in  Section \ref{ssec:dmd}. We describe a higher lag extension of DMD, which we call $\tau$-DMD, in Section \ref{sec:lagdmd}. We provide a DMD performance guarantee for unmixing deterministic signals in Section \ref{sec:genthm}; a corollary of that result in Section \ref{ssec:costhm} explains why DMD is particularly apt for unmixing multivariate mixtures of Fourier series such as  the ``eigen-walker'' model. We extend our analysis to  stationary, ergodic time series data in Section \ref{ssec:lagtau_random}. In Section \ref{sec:Sthm}, we provide results for the estimation error of the latent signals. We analyze the setting where the time series data matrix has randomly missing data in  Section \ref{sec:missing}. We validate our theoretical results with numerical simulations in Section \ref{sec:numerical simulations}. In  Section \ref{sec:realdata}, we describe how a time series matrix can be factorized using DMD to obtain a Dynamic Mode Factorization (DMF) involving the product of the DMD  estimate of the (column-wise normalized) mixing matrix and the coordinates, which represent the unmixed latent signals. We show how DMF can be applied to the cocktail party problem in \citet{choi2005blind} in Section \ref{sec:realdata} and how unmixing the latent series via DMF can help improve time series change point detection in  Section \ref{sec:changepoint}. We offer some concluding remarks in Section \ref{sec:conclusions}. The proofs of our results are deferred to the appendices. 

\subsubsection{Summary of Theorems}

{A contribution of this is a non-asymptotic finite sample performance analysis for the DMD and $\tau$-DMD algorithm in the setting  where the mixed deterministic signals or stationary, ergodic time series are approximately (or exactly) one- or higher lag uncorrelated. Our main results will concern the estimation errors of the mixing matrices. Theorem \ref{thm:shift} presents a general result with bounds for deterministic signals and all lags $\tau \geq 1$.  Corollary \ref{thm:cos} present bounds for the lag-one, deterministic case where the latent signals are cosines. Theorem \ref{thm:lagtautimeseries} generalizes Theorem \ref{thm:shift} to the setting where the latent signals are realizations of a stationary, ergodic time series. We present results for the estimation of the latent signals in Theorem \ref{thm:S_bound}, and extend the results to missing data in \ref{thm:missing}.}

% Section: Model and Algorithm %
\section{Model and Setup} \label{sec:model}

Suppose that, at time $t$, we are given a $p$ dimensional time series vector 
$$\xx_t = \begin{bmatrix} x_{1t} & x_{2t} & \ldots & x_{pt}\end{bmatrix}^T,$$
where an individual entry $x_{jt}$, for $ j = 1, 2, \ldots, p$, of $\xx_t$ is modeled as
\begin{equation}\label{eq:xjt}
x_{jt} = \sum_{i = 1}^k b_{ij} c_{it},
\end{equation}
and $b_{ij}$ is the $j^{th}$ entry of a $p$ dimensional vector $\bb_i$. Each $c_{it}$ is the $t^{th}$ entry of an $n$ dimensional vector $\cc_i$, and the $c_{it}$ are samples of a time series. Equation (\ref{eq:xjt}) can be succinctly written in vector form as 
\begin{equation} \label{eq:start model}
\xx_{t} = \sum_{i = 1}^k \bb_i c_{it} 
= B \, \begin{bmatrix} c_{1t} \\ \vdots \\ c_{kt}\end{bmatrix},
\end{equation}
where the $p \times k$ matrix $B$ is defined as 
$B = \begin{bmatrix} \bb_1 & \cdots & \bb_k\end{bmatrix}$.
We are given samples $\xx_1, \ldots, \xx_n$ corresponding to uniformly spaced time instances $t_1, \ldots t_n$. In what follows, without loss of generality, we assume that $t_i = i$. Let $X$ be the $p \times n$ matrix defined as
\begin{equation}
X = \begin{bmatrix} \xx_1 & \cdots & \xx_n\end{bmatrix}.
\end{equation}
We define the $n \times k$ matrix $C$ with columns $\cc_1, \ldots, \cc_k$ as
\begin{equation}\label{eq:C}
C^T = \left\{ \begin{bmatrix} & c_{1t} & \\ \cdots & \vdots & \cdots \\ &c_{kt} & \end{bmatrix}\right\}_{t = 1}^n.
\end{equation}
Consequently, we have that
\begin{equation}\label{eq:X1}
X = B\,C^T,
\end{equation}
where $C^T$ is the ``latent time series'' matrix given by (\ref{eq:C}). Equation (\ref{eq:X1}) reveals that 
the multivariate time series matrix $X$ is a linear combination of rows of the latent time series matrix. 

Suppose that for $i =1 , \ldots, k$, 
\begin{equation}\label{eq:qs}
\qq_i = \frac{\bb_i}{\|\bb_i\|_2} \textrm{ and } \sss_i = \frac{\cc_i}{\|\cc_i\|_2},
\end{equation}
and the matrices 
\begin{equation}\label{eq:QandS}
Q = \begin{bmatrix} \qq_1 & \cdots & \qq_k\end{bmatrix} \textrm{ and }  
S = \begin{bmatrix} \sss_1 & \cdots & \sss_k\end{bmatrix}.
\end{equation}
Then, from (\ref{eq:X1}), and from the definition of $Q$ and $S$, it can be shown that 
\begin{equation}\label{eq:our model}
X = Q\, D S^T 
\end{equation}
where, for $i = 1, \ldots, k$, 
\begin{equation}\label{eq:D}
D = \textrm{diag}\left(\ldots,\|\bb_i\|_2 \cdot \|\cc_i\|_2,\ldots \right).
\end{equation}
We will define 
\begin{equation}
d_i = \|\bb_i\|_2 \cdot \|\cc_i\|_2,
\end{equation}
and assume that, without loss of generality, the $d_i$ and hence the $\bb_i$, $\cc_i$, $\qq_i$, and $\sss_i$ are ordered so that 
\begin{equation}
d_1 \geq d_2 \geq \ldots \geq d_k > 0.
\end{equation}

Note that by construction, in (\ref{eq:our model}), the $k$ columns of the matrices $Q$ and $S$ have unit norm. In what follows, we assume that $Q$ and $S$ have linearly independent columns, that $k \leq p \leq n - 1$, that the columns of $S$ have zero mean, and that the columns of $Q$ are canonically \textit{non-random} and \textit{non-orthogonal}. Our goal in what follows is to estimate the columns of the matrices $Q$ and $S$.

\subsection{Dynamic Mode Decomposition (DMD)}\label{ssec:dmd}

From (\ref{eq:start model}), we see that the columns of $X$ represent a multivariate time series. We first partition the matrix $X$ into two $p \times n-1$ matrices 
\begin{equation}\label{eq:X0X1}
X_{(0)} = \begin{bmatrix} \xx_1 & \xx_2 & \cdots & \xx_{n - 1}\end{bmatrix} 
\textrm{ and }
X_{(1)} = \begin{bmatrix} \xx_2 & \xx_3 & \cdots & \xx_{n}\end{bmatrix}.
\end{equation}
We then compute the $p \times p$ matrix $\widehat{A}$ via the solution of the 
optimization problem
\begin{equation}\label{eq:dmd opt}
\widehat{A} = \argmin_{A \in \RR^{p \times p}} \left\|X_{(1)} - A X_{(0)}\right\|_F.
\end{equation}
The minimum norm solution to (\ref{eq:dmd opt}) is given by 
\begin{equation}\label{eq:Ahat}
\widehat{A} = X_{(1)} X_{(0)}^{+},
\end{equation}
where the superscript $^+$ denotes the Moore-Penrose pseudoinverse. Note that $\widehat{A}$ will be a non-symmetric matrix with a rank of at most $k$ because $X$, from which $X_{(1)}$ and $X_{(0)}$ are derived, has rank $k$ from the construction in (\ref{eq:our model}). Let
\begin{equation}\label{eq:eig dmd}
\widehat{A} = \widehat{Q} \widehat{\Lambda} \widehat{Q}^{+},
\end{equation}
be its eigenvalue decomposition. In (\ref{eq:eig dmd}), $\widehat{\Lambda} = \textrm{diag}(\lambda_1,\ldots, \lambda_k)$ is a $k \times k$ diagonal matrix, where the $\lambda_i$, ordered as $|\lambda_1| \geq |\lambda_2| \geq \ldots \geq |\lambda_k| > 0$, are the, possibly complex, eigenvalues of $\widehat{A}$ and $\widehat{Q}$ is a $p \times k$ matrix of, generically non-orthogonal, unit-norm eigenvectors, denoted by $\widehat{\qq}_i$. 

In what follows, we will refer to the computation of (\ref{eq:Ahat}) and the subsequent decomposition (\ref{eq:eig dmd}) as the DMD algorithm and we will show that under certain conditions, $\widehat{\qq}_i$ is close to $\qq_i$. 

% Section: Lag tau DMD %
\section{A Natural Generalization: \texorpdfstring{$\tau-$DMD}{Lags other than 1}} \label{sec:lagdmd}

We have just described the DMD algorithm at a lag of $1$. That is, we let $X_{(0)}$ and $X_{(1)}$ differ by one time-step. However, we might easily allow $X_{(0)}$ and $X_{(1)}$ to differ by $\tau$ time steps, and in certain settings, it may be advantageous to use $\tau > 1$. 

From (\ref{eq:start model}), we recall that the columns of $X$ represent a multivariate time series. We first partition the matrix $X$ into two $p \times n - \tau$ matrices:
\begin{equation}\label{eq:X0X1tau}
X_{(0)}^{\tau} = \begin{bmatrix} \xx_1 & \xx_2 & \cdots & \xx_{n - \tau}\end{bmatrix}
\textrm{ and } 
X_{(1)}^{\tau} = \begin{bmatrix} \xx_{1 + \tau} & \xx_{2 + \tau} & \cdots & \xx_{n}\end{bmatrix}.
\end{equation}
At this point, the procedure is identical to the DMD algorithm: we compute the $p \times p$ matrix $\widehat{A}(\tau)$ via the solution of the 
optimization problem
\begin{equation}\label{eq:dmdtauopt}
\widehat{A}_{\tau} = \argmin_{A \in \RR^{p \times p}} \left\|X_{(1)}^{\tau}  - A X_{(0)}^{\tau} \right\|_F,
\end{equation}
and the minimum norm solution to (\ref{eq:dmdtauopt}) is given by
\begin{equation}\label{eq:Ahattau}
\widehat{A}_{\tau}= X_{(1)}^{\tau}  \left(X_{(0)}^{\tau} \right)^{+}.
\end{equation}
Once again, let
\begin{equation}\label{eq:eig dmd tau}
\widehat{A}_{\tau} = \widehat{Q} \widehat{\Lambda} \widehat{Q}^{+},
\end{equation}
be its eigenvalue decomposition. In (\ref{eq:eig dmd tau}), $\widehat{\Lambda} = \textrm{diag}(\lambda_1,\ldots, \lambda_k)$ is a $k \times k$ diagonal matrix, where $|\lambda_1| \geq |\lambda_2| \geq \ldots \geq |\lambda_k| \geq 0$ are the (possibly complex) eigenvalues of $\widehat{A}_{\tau}$ and $\widehat{Q}$ is a $p \times k$ matrix of, generically non-orthogonal, unit-norm eigenvectors that are denoted by $\widehat{\qq}_i$. 

In what follows, we will refer to the computation of (\ref{eq:Ahattau}) and the subsequent decomposition (\ref{eq:eig dmd tau}) as the $\tau$-DMD algorithm. Note that the DMD algorithm is a special case of the $\tau$-DMD algorithm, and when we say `DMD' without any qualifiers, we mean the $\tau = 1$ setting.

\section{Performance Guarantee} \label{sec:genthm}

The central object governing the performance of the $\tau$-DMD algorithm is the lag-$\tau$ cross covariance matrix. Let the $k \times k$ lag-$\tau$ covariance matrix $L_{\tau}$ defined as
\begin{equation}\label{eq:Ltau}
\left[L_{\tau}\right]_{ij} = \sum_{l = 1}^{n} S_{i, l} S_{j, [l+\tau] \textrm{ mod } n}.
\end{equation}
Note that we can succinctly express $L_{\tau}$ as $L_{\tau} = S^T(P^{\tau} S)$ where $P$ is the matrix formed by taking the $n \times n$ identity matrix and circularly right shifting the columns by one. 

\subsection{Technical Assumptions}

We will require the following set of technical assumptions on the data. 

\begin{subequations}\label{eq:assumptions_general}
\begin{enumerate}
\item
Assume that $k$ is fixed, with 
\begin{equation} 
k \leq \min\left\{p, n - \tau\right\}
\end{equation}
\item
Assume that the $\qq_i$ are linearly independent, so that ${\sigma_1(Q)}/{\sigma_k(Q)}$ is a finite quantity:
\begin{equation}
1 \leq \frac{\sigma_1(Q)}{\sigma_k(Q)} < \infty.
\end{equation}
Here, $\sigma_i(Q)$ denotes the $i^{th}$ singular value of $Q$. Essentially, the conditioning of the $\qq_i$ is independent of $n$ and $p$. {Moreover, the $\qq_i$ are canonically \textit{non-random} and \textit{not necessarily orthogonal}.}
\item
Assume that 
\begin{equation} \label{eq:d_ratio}
\lim_{n \rightarrow \infty} \frac{d_1}{d_k} \nrightarrow \infty,
\end{equation}
i.e., that the limit of the ratio is finite. 
\item
Assume that columns of $S$ (the $\sss_i$) each have zero mean (the sum of each column is zero), and that they are linearly independent. Moreover, assume that there exists an $\alpha > 0$ such that 
\begin{equation} \label{eq:S_coherence}
\max_{i, j} |S_{ij}| = O\left(\frac{1}{n^{\alpha}}\right).
\end{equation}
I.e., the $\sss_i$ are not too sparse. 
\item
Assume that $\tau$ is small relative to $n$; i.e., that
\begin{equation}
\tau n^{-2 \alpha} \nrightarrow \infty
\textrm{ and }
n - \tau \approx n
\textrm{ for large $n$}. 
\end{equation}
\end{enumerate}
\end{subequations}
\begin{remark}
{Conditions 1, 2, and the first part of 4 are required for the data matrix to actually have rank $k$. I.e., if there are $k$ latent signals, we need the columns of $Q$ to be linearly independent and we need the signals to be linearly independent to recover all $k$ signals and the $k$ columns of $Q$ and not linear combinations thereof. We need at least as many linear combinations and samples as there are signals to recover the signals. Moreover, the linear independence and full column rank conditions yield that $Q$ and $S$ are unique, and hence can (in principle) be estimated uniquely up to a sign or phase shift. Note that for a rank $k$ matrix, there are many different possible factorizations, but our results here will identify when the specific $Q$ and $S$ matrices can be recovered. Condition 3 ensures that, in the limit, we can recover all $k$ signals. Intuitively, if the ratio (\ref{eq:d_ratio}) diverged, the data matrix would eventually have a numerical rank smaller than $k$, and the smallest signal would look like noise relative to the largest. Finally, the second part of condition 4 ensures that the latent signals are sufficiently dense, or that they are not very transient. That is, the signals are not something like a spike. Condition 4 is purely technical and is needed for the proofs of the performance bounds. Finally, condition 5 is technical, and ensures that each of $X_{(0)}^{\tau}$ and $X_{(1)}^{\tau}$ contain enough information. }
\end{remark}

\subsection{Deterministic Signals}

We now establish a recovery condition for the setting where $\cc_i$ in (\ref{eq:C}) are deterministic. 

\begin{remark}
In the following result and in all subsequent results, there is an ambiguity or mismatch between the ordering of the $\qq_i$, $\sss_i$, $d_i$, and $\left[L_{\tau}\right]_{ii}$ with that of the $\widehat{\qq}_j$ and $\lambda_j$. Formally, there exists a permutation $\sigma(i)$ that reorders the $\widehat{\qq}_j$ and $\lambda_j$ to correspond to the $\qq_i$ and other quantities, such that the error is minimal. In the statement of our results, without loss of generality, we will assume that $\sigma(i) = i$, i.e., that it is the identity permutation.  
\end{remark}

\begin{theorem}[$\tau$-lag DMD] \label{thm:shift}
For $X$ as in (\ref{eq:our model}) and $L_{\tau}$ defined as in (\ref{eq:Ltau}), suppose that the conditions in (\ref{eq:assumptions_general}) hold. Further suppose that 
\begin{subequations}
\begin{equation}
\lim_{n \to \infty} \left|\left[L_{\tau}\right]_{ii}\right| \nrightarrow 0.
\end{equation}
 \textrm{Moreover, assume that for $i \neq j$ we have that}
 \begin{equation}
 \left|\left[L_{\tau}\right]_{ij}\right| = O(f(n)) \textrm{ and } \left|\sss_i^T \sss_j\right| = O(f(n))
 \end{equation}
for some $f(n)$ such that 
$\lim_{n \rightarrow \infty} f(n) = 0$.
\end{subequations}

\noindent
a) Then, assuming that 
$p_i$ is given by 
\begin{equation} \label{eq:q_sign}
p_i = \textrm{sign}\left(\widehat{\qq}_i^T \qq_i\right),
\end{equation}
we have that 
\begin{subequations}
\begin{equation}\label{eq:generalboundshift}
\sum_{i = 1}^k \left\|\widehat{\qq}_i - p_i \qq_i\right\|_2^2 = \! O\left(\left[\dfrac{d_1}{d_k}\right]^2 \cdot \dfrac{k^7}{\delta_{L}^2} \cdot \left[f^2(n) + \tau n^{-2\alpha} \right]\right),
\end{equation}
\textrm{where $\delta_{L}$ is given by}
\begin{equation} \label{eq:delta_L_tau}
{\delta_{L} = \min_{i \neq j} \left|\left[L_{\tau}\right]_{ii} - \left[L_{\tau}\right]_{jj}\right|.}
\end{equation}

\noindent
b) Moreover, for each $\left[L_{\tau}\right]_{ii}$, we have that 
\begin{equation}\label{eq:evalboundshift}
\left|\left[L_{\tau}\right]_{ii} - \lambda_i\right|^2 = O\left(\left[\dfrac{d_1}{d_k}\right]^2 \cdot {k^6}\cdot \left[f^2(n) + \tau n^{-2\alpha} \right]\right).
\end{equation}
\end{subequations}
\end{theorem}

Note that the bound (\ref{eq:generalboundshift}) depends on $\delta_{L}$: if two of the signals have identical lag-$\tau$ autocorrelations, the bound becomes trivial and the signals may not be able to be unmixed. 

{Moreover, this result is entirely in terms of the latent signals, $\sss_i$: $f(n)$ is the lag-$1$ cross correlation decay rate, $\alpha$ governs the sparsity/density of the signals, and $d_i$ is the magnitude of each signal. We have specified conditions on the latent signals such that they may be unmixed. Of course, without knowledge of the latent signals, these bounds are not computable. Noting that $\delta_{L}$ is a function of $\tau$, we anticipate that some values of $\tau$ would lead to better results than others: we will demonstrate this behavior numerically in Section \ref{sec:numerical simulations}.}

\subsection{Application of Theorem \ref{thm:shift}: DMD Unmixes Multivariate Mixed Fourier Series}\label{ssec:costhm}
Consider the setting where $c_{it}$ in (\ref{eq:xjt}) is modeled as
\begin{equation} \label{eqn:cos}
c_{it} = \cos\left(\omega_i t + \phi_i\right). 
\end{equation}
The $x_{it}$ is thus a linear mixture of Fourier series. This model frequently comes up in many applications such as the eigenwalker model for human motion:
\citet[Equations (1) and (2)]{troje2002decomposing}, \citet{troje2002little} and \citet[Equations (1) and (2)]{unuma1995fourier}.

This model fits into the framework of Theorem \ref{thm:shift} via an application of Corollary \ref{thm:cos} below. This implies the DMD modes will correctly correspond to the non-orthogonal mixing modes. Using PCA on the data matrix in this setting would recover orthogonal modes that would be linear combinations of the latent non-orthogonal dynamic modes.

\begin{corollary}[Mixtures of Cosines] \label{thm:cos}
Assume that the $\cc_i$ are given by (\ref{eqn:cos}), that the $p_i$ are given by (\ref{eq:q_sign}), and that we apply DMD with $\tau = 1$. Then we have that 
\begin{subequations}
\begin{equation} \label{eq:cos_bound}
\sum_{i = 1}^k \left\|\widehat{\qq}_i - p_i \qq_i\right\|_2^2 = O\left(\left[\dfrac{d_1}{d_k}\right]^2 \cdot \dfrac{k^7}{\delta_{L}^4} \cdot \dfrac{1}{n}\right),
\end{equation}
\textrm{where}
\begin{equation} \label{eq:delta_cos}
\delta_{L} = \min_{i \neq j} \left|\cos \omega_i - \cos \omega_j\right|,
\end{equation}
and that for each $\omega_i$, we have that
\begin{equation}
\left|\cos \omega_i - \lambda_i\right|^2 = O\left(\left[\dfrac{d_1}{d_k}\right]^2  \cdot \dfrac{k^6}{n}\right).
\end{equation}
\end{subequations}
\end{corollary}

Corollary \ref{thm:cos} explains why DMD successfully unmixes the eigenwalker data in Figure \ref{fig:eigwalk_intro}. In that setting, PCA does not succeed because it returns an orthogonal matrix as an estimate of the non-orthogonal  mixing matrix. The ability of DMD to reliably unmix non-orthogonally mixed multivariate Fourier series, and the fact that the eigenvalues are cosines of the frequencies, provides some context for the statement that DMD is a spectral algorithm where the eigen-spectra reveal information on Fourier spectra \citep{rowley2009spectral}.

Note that by Theorem \ref{thm:shift}, we require that the lag-$1$ autocorrelations are distinct. In this case, it is equivalent to requiring that the cosines have distinct frequencies. {In the notation of Theorem \ref{thm:shift}, we have that $\alpha = 1/2$ and $f(n) = 1/\sqrt{n}$.}

\subsection{Extensions of Theorem \ref{thm:shift}: Stationary, Ergodic Time Series}\label{ssec:lagtau_random}

We now consider the setting where  $c_{it}$ are elements of a stationary, ergodic time series and the $\cc_i$, thus formed; {we say that a process is stationary and ergodic when its statistical properties do not change over time, and when they can be estimated from a sufficiently long realization. We point the reader to \cite[Ch.~2.3,~15.4]{jonathan2008time} for formal definitions of these terms.} Consider the matrix
\begin{equation}\label{eq:Ltautimeseries}
\EE \left[L_{\tau}\right]_{ij} = \EE \left[ S_{i, l} S_{j, [l+\tau] \textrm{ mod } n}\right].
\end{equation}

When $\EE {L_{\tau}}$ is diagonal, then $\tau$-DMD asymptotically unmixes the time series, as expressed in the Theorem below. We will require the assumptions from (\ref{eq:assumptions_general}), with the following updates:
\begin{subequations} \label{eq:assumptions_arma}
\begin{enumerate}
\item
Assume that the $\bb_i$, $\cc_i$, $\qq_i$, and $\sss_i$ are ordered so that 
\begin{equation}
\EE d_1 \geq \EE d_2 \geq \ldots \geq \EE d_k > 0,
\end{equation}
where $\EE d_i = \|\bb_i\|_2 \cdot \EE \|\cc_i\|_2$.

\item
Assume that 
\begin{equation} \label{eq:random_d1dk}
\lim_{n \rightarrow \infty} \frac{\EE d_1}{\EE d_k} \nrightarrow \infty,
\end{equation}
i.e., that the limit of the ratio is finite.
\end{enumerate}
\end{subequations}

\begin{theorem}[Stationary, Ergodic Time Series at Lag $\tau$] \label{thm:lagtautimeseries}
Suppose that the conditions in (\ref{eq:assumptions_arma}) hold, in addition to conditions (1, 2, 4, 5) from (\ref{eq:assumptions_general}). 
\begin{subequations}
\begin{equation}\label{eqn:tauts}
1 \leq \tau \leq n^{\frac{r}{2 (r - 2)}},
\end{equation}
for some value of $r \geq 4$. Let the $\cc_i$ be as described above, and let $\EE L(\tau)$ be as defined in (\ref{eq:Ltautimeseries}). Assume that $\EE \left[L_{\tau}\right]_{ii} \neq 0$, $\EE \left[L_{\tau}\right]_{ij} = 0$, and $\EE \sss_i^T \sss_j = 0$. Then, we have that

\noindent
a) For some $\epsilon > 0$ and $r \geq 4$, we have that 
\begin{equation} \label{eqn:flagtau}
f(n) = o\left(\left(\log n\right)^{2 / r} \left(\log \log n\right)^{(1 + \epsilon) 2 / r} n^{-1/2}\right).
\end{equation}
Then, {there exists a constant $c$ such that}
\begin{equation}
\left|\left[L_{\tau}\right]_{ij}\right|= O(f(n)) \textrm{ and } \left|\sss_i^T \sss_j\right| = O(f(n))
\end{equation}
with probability at least
\begin{equation}\label{eqn:lagtautsprob}
1 - c \left(\left[\log n \left(\log \log n\right)^{1 + \epsilon}\right]^{-1}\right).
\end{equation}

\noindent
b) Then we have that 
$|d_i - \EE d_i| \leq f(n) [1+o(1)]$ for $i = 1, \ldots, k$,
with probability (\ref{eqn:lagtautsprob}).

\noindent
c) For
% an $\alpha$ defined as in (\ref{eqn:Salpha}) and 
$p_i$ given by (\ref{eq:q_sign}), we have that 
\begin{equation} \label{eq:lagtauloss}
\sum_{i = 1}^k \left\|\widehat{\qq}_i - p_i \qq_i\right\|_2^2 = \! O\left(\left[\dfrac{\EE d_1}{\EE d_k}\right]^2 \cdot \dfrac{k^7}{\delta_{L}^2} \cdot \left[f^2(n) + \tau n^{-2\alpha} \right]\right),
\end{equation}
\textrm{where $\delta_{L}$ is given by}
\begin{equation}
\delta_{L} = \min_{i \neq j} \left|\EE \left[L_{\tau}\right]_{ii} - \left[L_{\tau}\right]_{jj}\right|,
\end{equation}
with probability (\ref{eqn:lagtautsprob}).

\noindent
d) Moreover, for each $\EE L_{ii}(\tau)$, we have that 
\begin{equation}\label{eq:lagtaueig}
\left|\EE \left[L_{\tau}\right]_{ii} - \lambda_i\right|^2 = O\left(\left[\dfrac{\EE d_1}{\EE d_k}\right]^2 \cdot {k^6}\cdot \left[f^2(n) + \tau n^{-2\alpha} \right]\right),
\end{equation}
with probability (\ref{eqn:lagtautsprob}). 

\end{subequations}
\end{theorem}

If the $\cc_i$ are samples from a stationary, ergodic ARMA process, we may simplify the results of Theorem \ref{thm:lagtautimeseries} slightly. 
\begin{corollary}[ARMA Processes at Lag $\tau$] \label{thm:armacorrtau}
Assume that the $\cc_i$ are samples from an ARMA process. Then (\ref{eqn:tauts}) may be replaced with
$1 \leq \tau \leq \left[\log n\right]^a$,
for some $a > 0$, and (\ref{eqn:flagtau}) may be replaced with 
$f(n) = o\left(\left(\log \log n / n\right)^{1/2}\right)$.
\end{corollary}

The iterated logarithmic rate in our error bounds and accompanying probability, are consequences of the classical time series results in \citet{hong1982autocorrelation}. {Here, we have stated a result that is similar in spirit to that for SOBI, given in \citet{miettinen2016separation}. Our result says that time series $\sss_i$ that are uncorrelated at lags $1$ and $0$ can be unmixed, provided that they are not sparse. The result for SOBI requires uncorrelatedness at all integral lags, and states an asymptotic distributional result; our result relies on looser assumptions, and is a finite sample guarantee. It should be noted that at the expense of using a single lag, our result is slightly weaker than the $1/\sqrt{n}$ convergence described in \citet[Theorem~1]{miettinen2016separation}.}

% Section: S Theorem %
\section{Estimating the temporal behavior: \texorpdfstring{$S$}{S}}\label{sec:Sthm}

We now establish a recovery condition for deterministic $\sss_i$. 

\begin{theorem}[Extending the bounds to $S$] \label{thm:S_bound}
Assume that the conditions of Theorem \ref{thm:shift} hold for a lag $\tau$ with a bound $\epsilon_{d, v}^2$ for the squared estimation error of the $\qq_j$. Moreover, assume that
$k d_1^2 \epsilon_{d, v}^2 < d_k^2$. 
Then, given an estimate of the top $k$ left eigenvectors of $\widehat{A}$, denoted by the rows of the matrix $\widehat{Q^{+}}$, let $\widehat{S}$ be formed by normalizing the columns of $\left(\widehat{Q^{+}} X\right)^T$. The columns of $\widehat{S}$ are denoted by $\widehat{\sss}_i$, and let 
$p_i = \textrm{sign}\left(\sss_i^T \widehat{\sss}_i\right)$. 
Then, we have that 
\begin{equation}\label{eq:S_err}
\sum_{i = 1}^k \left\|\widehat{\sss}_i - p_i \sss_i\right\|_2^2 = O\left(k \left[\frac{d_1}{d_k}\right]^2 \epsilon_{d, v}^2\right).
\end{equation}
\end{theorem}

This result translates the results for the mixing matrix $Q$ to the estimation of the signals $S$. For the practitioner intending to estimate the latent \emph{signals} instead of the mixing matrix, this final result has a greater utility.  

\subsection{Applications of Theorem \ref{thm:S_bound}: Cosines}

As we did for Theorem \ref{thm:shift}, we may restate Theorem \ref{thm:S_bound} for the cosine model.

\begin{corollary}[Cosines] \label{thm:cos_S}
Assume that the $\cc_i$ are given by (\ref{eqn:cos}), that we apply DMD with $\tau = 1$, and that $p_i = \textrm{sign}\left(\sss_i^T \widehat{\sss}_i\right)$. Then we have that 
\begin{equation} \label{eq:cos_bound_S}
\sum_{i = 1}^k \left\|\widehat{\sss}_i - p_i \sss_i\right\|_2^2 = O\left(\left[\dfrac{d_1}{d_k}\right]^4 \cdot \dfrac{k^8}{\delta_{L}^4} \cdot \dfrac{1}{n}\right),
\end{equation}
where
$\delta_{L} = \min_{i \neq j} \left|\cos \omega_i - \cos \omega_j\right|$.
\end{corollary}

% Section: Missing Data %
\section{Missing Data Analysis} \label{sec:missing}

We now consider the randomly missing data setting. We assume that the data is modeled as  
\begin{equation}
\widetilde{X} = X \odot M = \left(Q D S^T\right) \odot M,
\end{equation}
where $M$ is a masking matrix, whose entries are drawn uniformly at random: 
\begin{equation}
M_{i, j} = \left\{ 
\begin{array}{ll} 
1 & \text{ with probability } q,\\ 
0 & \text{ with probability } 1 - q.
\end{array}\right.
\end{equation}
The notation $\odot$ represents the Hadamard or element-wise matrix product. {Essentially, we replace unknown entries with zeros, as is done in the compressed sensing literature \cite{candes2011probabilistic, ravikumar2010high, nadakuditi2014optshrink}.}

\subsection{The tSVD-DMD algorithm}
A natural, and perhaps the simplest, choice to `fill-in' the missing entries in $\widetilde{X}$ is to use a low-rank approximation, also known as a truncated SVD \citep{davenport2016overview, eckart1936approximation}. That is, given $\widetilde{X}$, we compute the SVD 
$\widetilde{X} = \widehat{U} \widehat{\Sigma} \widehat{V}^T$,
and then the rank-$k$ truncation 
\begin{equation}
\widehat{X}_k = \sum_{i = 1}^k \widehat{\sigma}_i \widehat{\uu}_i \widehat{\vv}_k^T,
\end{equation}
where the columns of $\widehat{U}$ and $\widehat{V}$ are the $\widehat{\uu}_i$ and $\widehat{\vv}_i$, respectively, and the $\widehat{\sigma}_i$ are the non-zero entries of $\widehat{\Sigma}$. In what follows, $\uu_i$, $\vv_i$, and $\sigma_i$ will denote the singular vectors and values of $X$. We assume that the number of sources $k$ is known apriori.

After `filling-in' the missing entries of $\widetilde{X}$ and computing $\widehat{X}_k$, we may apply the $\tau$-DMD algorithm to $\widehat{X}_k$. If $\widehat{X}_k$ has columns 
$\widehat{X}_k = \begin{bmatrix} \widehat{\xx}_1 & \widehat{\xx}_2 & \cdots &  \widehat{\xx}_n\end{bmatrix}$,
we may define 
\begin{equation}\label{eq:X0X1hattau}
\widehat{X}_{(0)}^{\tau} = \begin{bmatrix} \widehat{\xx}_1 & \widehat{\xx}_2 & \cdots & \widehat{\xx}_{n - \tau}\end{bmatrix}
\textrm{ and }
\widehat{X}_{(1)}^{\tau} = \begin{bmatrix} \widehat{\xx}_{1 + \tau} & \widehat{\xx}_{2 + \tau} & \cdots & \widehat{\xx}_{n}\end{bmatrix}.
\end{equation}
We have dropped the $k$-dependence for clarity. Then, we may define 
\begin{equation}\label{eq:Ahattaumissing}
\widetilde{A}_{\tau}= \widehat{X}_{(1)}^{\tau}  \left(\widehat{X}_{(0)}^{\tau} \right)^{+},
\end{equation}
and take an eigenvalue decomposition:
\begin{equation}\label{eq:eig dmd tau missing}
\widetilde{A}_{\tau} = \widehat{Q} \widehat{\Lambda} \widehat{Q}^{+}.
\end{equation}
For the sake of naming consistency, we will refer to this procedure as the tSVD-DMD algorithm. 

\subsection{Assumptions}
We now provide a DMD recovery performance guarantee.  Before stating the result, we require some definitions and further conditions. In addition to the previous assumptions about $S$, the $d_i$, the relative values of $k$, $n$, $p$, and $\tau$, and the linear independence of the $\qq_i$, we require the following conditions that augment (\ref{eq:assumptions_general}). For clarity and conciseness in what follows, we define the constant
\begin{subequations} \label{eq:missing_defs}
\begin{equation} 
\gamma = \frac{n^{2 \alpha} p^{2 \beta}}{d_1^2 k^2},
\end{equation}
and the quantities
\begin{equation}
\begin{split}
g(n, p, k, q) = O\biggl(\sqrt[4]{q (1 - q)} d_1 k \times \max\left\{n^{1/4 - \alpha } p^{1/4 - \beta}, n^{-\alpha}, p^{-\beta}\right\}\biggr),
\end{split}
\end{equation}
\begin{equation}
\delta_{\sigma, q} = \min_{i = 1, 2, \ldots, k - 1}\left\{q \sigma_k, q^2 \sigma_k^2, q^2 \sigma_{i} (\sigma_{i} - \sigma_{i + 1}), q \left(\sigma_i - \sigma_{i + 1}\right)\right\},
\end{equation}
and
\begin{equation}
\delta_{\sigma} = \min_{i = 1, 2, \ldots, k - 1}\left\{\sigma_k, \sigma_k^2, \sigma_{i} (\sigma_{i} - \sigma_{i + 1}), \sigma_i - \sigma_{i + 1}\right\}.
\end{equation}
\end{subequations}
{The quantity $g(n, p, k, q)$ comes from bounding the size of $\left(\widetilde{X} - \EE \widetilde{X}\right)$, motivated by the approach taken in \citet{nadakuditi2014optshrink} for handling missing data. The quantities $\delta_{\sigma}$ and $\delta_{\sigma, q}$ come from applications of the results in \citet[Corollary~20, Theorem~23]{o2018random}. The details of how these quantities arise and are used are deferred to the proof of Theorem \ref{thm:missing}, given in Appendix \ref{sec:missing_proof}. }

Then, we require:
\begin{subequations} \label{eq:assumptions_missing}
\begin{enumerate}
\item
Assume that there is a $\beta > 0$ such that
\begin{equation} \label{eq:Q_coherence}
\max_{1 \leq i \leq p, 1 \leq j \leq k} |Q_{i, j}| = O\left(p^{-\beta}\right).
\end{equation}
I.e., the $\qq_i$ are not too sparse; this condition is exactly analogous to that for the $\sss_i$, where we used the parameter $\alpha$. 

\item
Assume that as $p$ and $n$ grow,
\begin{equation}
\frac{1}{\delta_{\sigma, q}}, \frac{q \sigma_{1}}{\delta_{\sigma, q}}, \frac{1}{\gamma \delta_{\sigma, q}} \nrightarrow \infty.
\end{equation}

\item
Assume that 
\begin{equation} \label{eq:missing_pn1}
\lim_{p, n \rightarrow \infty} d_1 \cdot \max\left\{n^{1/4 - \alpha } p^{1/4 - \beta}, n^{-\alpha}, p^{-\beta}\right\} = 0,
\end{equation}
but that 
\begin{equation}\label{eq:missing_pn2}
\lim_{p, n \rightarrow \infty} g(n, p, k, q)^2 \gamma \neq 0.
\end{equation}
\end{enumerate}
\end{subequations}

Condition (\ref{eq:Q_coherence}), along with the analogous condition for the $\sss_i$ given in (\ref{eq:S_coherence}), corresponds to the low coherence condition in the matrix completion literature \citep[Section 5.2]{davenport2016overview}. I.e., we require that the data matrix is sufficiently dense. Moreover, (\ref{eq:missing_pn1}) and (\ref{eq:missing_pn2}) imply that the $\sss_i$ and $\qq_i$ have values of $\alpha$ and $\beta$ that are at least $1/4$ {(and less than $1/2$, by definition). For example, if we generate a matrix $Q$ by uniformly drawing $k$ vectors from the sphere in $\RR^p$ and setting these as the columns, and let $S$ be comprised of cosines as in (\ref{eqn:cos}), we would anticipate that $\alpha = \beta = 1/2$. In this case, if $d_1$ is not increasing, we would have that $g(n, p, k, q) = O\left(\sqrt{q} k / \sqrt[4]{p n}\right)$. }

Given these assumptions, if we apply the tSVD-DMD algorithm to $\widetilde{X}$, we have the following result for the estimation of the eigenvectors $\qq_j$ and eigenvalues $\lambda_i$. 

\subsection{Main result}
\begin{theorem}[Missing Data Recovery Guarantee] \label{thm:missing}
Let the assumptions of Theorem \ref{thm:lagtautimeseries} hold, with a bound $\epsilon_{d, v}^2$ for the squared estimation error of the $\qq_i$ and a bound $\epsilon_{d, e}^2$ for the squared error for the individual eigenvalues. Let the conditions in (\ref{eq:assumptions_missing}) hold, let $a > 1$, {and let $c_0, c_1, c_2 > 0$ be some universal constants.} 
 
\noindent
a) Then, if $L_{\tau}$ is defined in (\ref{eq:Ltau}), $\delta_L$ is defined in (\ref{eq:delta_L_tau}), and $p_i$ is defined in (\ref{eq:q_sign}), 
\begin{equation}\label{eq:missingloss}
\begin{split}
\sum_{i = 1}^k \left\|\widehat{\qq}_i - p_i \qq_i\right\|_2^2 = O\left(\frac{\tau}{q^2} a^2 \left(g(n, p, k, q) \right)^2\frac{\sigma_1^2}{\delta_{\sigma}^2} \frac{k^8}{\delta_{L}^2} + \epsilon_{d, v}^2 \right),
\end{split}
\end{equation}
with probability at least
\begin{equation} \label{eq:missingprob}
\begin{split}
1 - c_1 \left(k^2 \cdot 81^k \exp\left(-\left(1 - \frac{1}{a}\right)^2 c_0 \gamma \frac{\tau \left(g(n, p, k, q)\right)^2}{16}\right)\right)
- c_2\left(k^2 \cdot 9^k \exp\left(-c_0 \gamma \frac{\delta_{\sigma, q}}{64}\right)\right).
\end{split}
\end{equation}

\noindent
b) For each $[L_{\tau}]_{ii}$, we have that 
\begin{equation} \label{eq:eigenvalue_missing}
\begin{split}
\left|[L_{\tau}]_{ii} - \lambda_i\right|^2 = O\left(\frac{\tau}{q^2} a^2 \left(g(n, p, k, q) \right)^2\frac{\sigma_1^2}{\delta_{\sigma}^2} {k^7} + \epsilon_{d, e}^2 \right),
\end{split}
\end{equation}
with probability at least (\ref{eq:missingprob}). 

\end{theorem}

{Note that Theorem \ref{thm:missing} indicates that the dependence of the squared estimation error on $q$ is $O(q^{-3/2})$ for $q$ close to $0$. Moreover, for data such that $d_1$, $\sigma_1$, $\delta_{\sigma}$ and $\delta_{L}$ are not changing with $n$; $Q$ has dense, linearly independent columns; and such that $k$ and $p$ are fixed, the right-hand sides of (\ref{eq:missingloss}) and (\ref{eq:eigenvalue_missing}) behave like
$O\left(q^{-3/2} n^{1/2 - 2 \alpha}\right)$
with probability at least 
$1 - c_1 \left(\exp\left(-c_3 \sqrt{n}\right)\right) - c_2 \left(\exp\left(-c_4 n q\right)\right),$
for some constants $c_1, c_2, c_3, c_4$. Indeed, if the $\cc_i$ are cosines, given by (\ref{eqn:cos}), we have that $\alpha = 1/2$, so that we have a rate of $O\left(q^{-3/2} n^{-1/2}\right)$.}

% Section: Simulations %
\section{Numerical simulations}\label{sec:numerical simulations}

In this section, we provide a numerical verification of the theorems we have presented. That is, we generate data, compute the quantities described in the theorems, and observe that these quantities satisfy the bounds presented in the theorems. We recall that one of the contributions of this work and the intention of this work is to demonstrate that DMD is a source separation algorithm in disguise. Our goals are not to compete with the state-of-the art in source separation, rather, this work seeks to provide a new analysis and understanding of the DMD algorithm.

There are two main objects of interest: the error in estimating the eigenvectors $\qq_i$, and the error in estimating the eigenvalues $\lambda_i$. In the deterministic, fully observed setting, the error in estimating $\sss_i$ is also of interest. In what follows, unless otherwise noted, we fix $p = 100$ and $k = 2$, and vary $n$. We fix the mode magnitudes at $d_1 = d_2 = 1$. We also generate dense, non-orthogonal $\qq_i$ by sampling from the sphere in $\RR^p$. Equivalently, we sample from the multivariate normal distribution $\Nn\left(\bzr_p, \Ii_p\right)$ and normalize the resulting vector to have unit $\ell_2$ norm. 

We first verify the deterministic error bounds for the cosine model with the DMD algorithm: i.e., Theorem \ref{thm:shift} and Corollary \ref{thm:cos}, as well as Theorem \ref{thm:S_bound} and Corollary \ref{thm:cos_S}. These verifications are presented in Figure \ref{fig:thmgeneral_cos}. We let the columns of $C$ be equal to $\cc_{i, t} = \cos \left(\omega_i t\right)$. We consider two sets of frequencies: $\omega_1 = 0.25$ and $\omega_2 = 0.5$, as well as $\omega_1 = 0.25$ and $\omega_2 = 2$. We see that as expected, the squared estimation errors for the eigenvalues $\lambda_i$, eigenvectors $\qq_i$, and the $\sss_i$ are bounded by $O(1/n)$. Moreover, the role of $\delta_{L}$ (defined in (\ref{eq:delta_cos})) is visible, as $\omega_2 = 2$ leads to a lower error relative to $\omega_2 = 0.5$ when estimating the $\qq_i$ and $\sss_i$. As expected, the non-zero eigenvalues are equal to $\cos \omega_i$.

We next consider the $\tau$-DMD algorithm, and verify Theorems \ref{thm:shift} and \ref{thm:lagtautimeseries}, as well as Corollary \ref{thm:armacorrtau}. We generate the columns of $C$ as independent, length $n$ realizations of AR(2) processes. That is, $\cc_1$ is a realization of an AR(2) process with parameters $(0.2, 0.7)$, and $\cc_2$ is also a realization of an AR(2) process with parameters $(0.3, 0.5)$. We compare operating at lags $\tau = 1$ and $\tau = 2$, and average over $200$ realizations. Our results appear in Figure \ref{fig:arma}. Note that for a given lag, the non-zero eigenvalues are expected to equal the autocorrelation of the $\cc_i$ at that lag; invoking the role of $\delta_{L}$ once again, we observe that the $\qq_i$ are better estimated at a lag of $\tau = 2$, as the lag-$2$ autocorrelations are higher and more separated than the lag-$1$ values. As expected, the squared estimation errors are bounded by $O(\log \log n / n)$.

{Finally, we consider the tSVD-DMD algorithm in the presence of missing data, and verify Theorem \ref{thm:missing}. Here, we fix $p = 2000$ and let $d_1 = 2$ and $d_2 = 1$. We let the columns of $C$ be equal to $\cc_{i, t} = \cos\left( \omega_i t\right)$, for $\omega_1 = 0.25$ and $\omega_2 = 2.0$. Our results are averaged over $50$ trials. We consider the effects of varying the entry-wise observation probability $q$ (for $n = 10^4$) in Figure \ref{fig:missing_fix_n}, and the effects of varying $n$ (for $q = 0.1$) in Figure \ref{fig:missing_fix_q}. As expected, we see that the squared estimation error of the eigenvectors decays at a rate bounded by $O(1/\sqrt{n})$ for fixed $q$ and like $O(q^{-3/2})$ for fixed $n$ when using the truncated SVD as a preprocessing step. The squared estimation error of the eigenvalues is bounded by the same rates. It is likely that these rates are somewhat conservative. Note that the error of DMD without the SVD is orders of magnitude larger than it is with the SVD, and does not exhibit significant decay with increasing $n$ or $q$.  }

\begin{figure*}[htb]
\begin{multicols}{3}
\begin{minipage}[b]{\linewidth}
\centering
\centerline{\includegraphics[width = \textwidth]{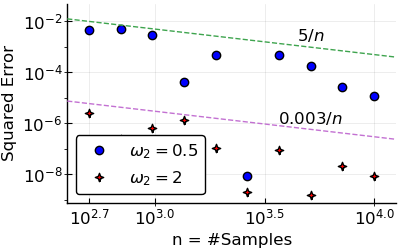}}
\subcaption{The squared estimation error of $\widehat{Q}$ as in (\ref{eq:cos_bound}). }
\end{minipage}
\begin{minipage}[b]{\linewidth}
\centering
\centerline{\includegraphics[width = \textwidth]{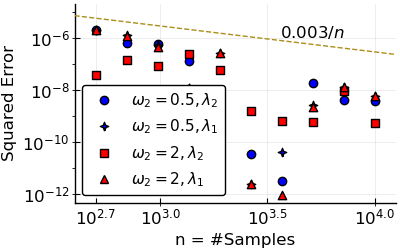}}
\subcaption{The squared estimation error of the eigenvalues $\widehat{\lambda}_i$ as in (\ref{eq:eig dmd tau}).}
\end{minipage}
\begin{minipage}[b]{\linewidth}
\centering
\centerline{\includegraphics[width = \textwidth]{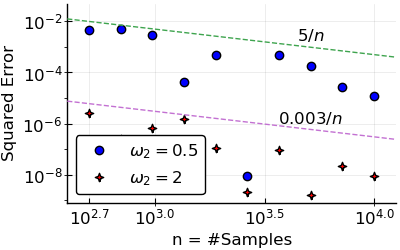}}
\subcaption{The squared estimation error of $\widehat{S}$ as in (\ref{eq:S_err}). }
\end{minipage}
\end{multicols}
\vspace{-0.7cm}
\caption{Here, we verify Theorem \ref{thm:shift} and Corollary \ref{thm:cos}, as well as Theorem \ref{thm:S_bound} and Corollary \ref{thm:cos_S}. We simulate from model (\ref{eq:start model}) with a rank $2$ cosine signal, first using $\omega_1 = 0.25$ and $\omega_2 = 0.5$, and second using $\omega_2 = 2$. We fix $p = 100$ and use a non-orthogonal $Q$, and apply DMD with $\tau = 1$. Note that as $\omega_1$ is fixed,  $\omega_2 = 2$ leads to a lower error relative to $\omega_2 = 0.5$, due to the greater separation of the frequencies: the error is proportional to $\frac{1}{\left|\omega_1 - \omega_2\right|}$. We also plot lines above the samples indicating that the error is bounded by $O(1/n)$. }
\label{fig:thmgeneral_cos}
\end{figure*}

\begin{figure*}[htb]
\begin{multicols}{3}
\begin{minipage}[b]{\linewidth}
\centering
\centerline{\includegraphics[width = \textwidth]{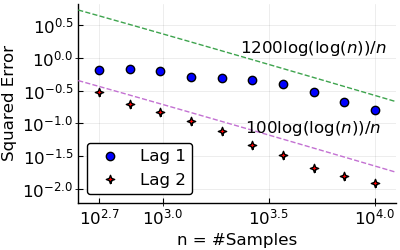}}
\subcaption{The squared estimation error of $\widehat{Q}$ as in (\ref{eq:lagtauloss}). }
\end{minipage}
\begin{minipage}[b]{\linewidth}
\centering
\centerline{\includegraphics[width = \textwidth]{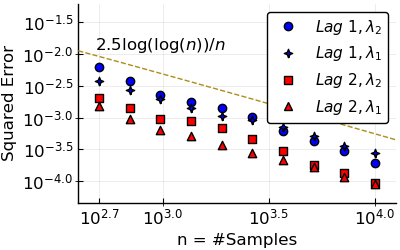}}
\subcaption{The squared estimation error of the eigenvalues $\widehat{\lambda}_i$ as in (\ref{eq:lagtaueig}). }
\end{minipage}
\begin{minipage}[b]{\linewidth}
\centering
\centerline{\includegraphics[width = \textwidth]{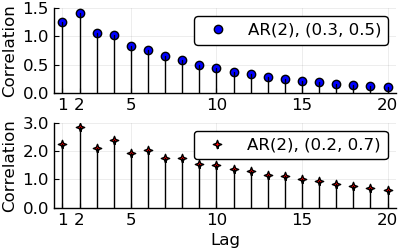}}
\subcaption{The autocorrelation function of the processes in $C$. }
\end{minipage}
\end{multicols}
\vspace{-0.7cm}
\caption{Here, we verify Theorems \ref{thm:shift} and \ref{thm:lagtautimeseries}, as well as Corollary \ref{thm:armacorrtau}. We simulate from model (\ref{eq:start model}) with a rank $2$ signal. The signals in $C$ are drawn as realizations from AR(2) processes, the first with parameters $[0.3, 0.5]$ and the second with parameters $[0.2, 0.7]$. We fix $p = 100$ and use a non-orthogonal $Q$. The lag-2 DMD algorithm leads to a lower eigenvector loss, as expected, since the autocorrelations at lag-$2$ are more separated from each other and from zero than they are at a lag of $1$. We also plot lines above the samples indicating that the error is bounded by $O(\log \log n / n)$.}
\label{fig:arma}
\end{figure*}

\begin{figure*}[htb]
\begin{multicols}{2}
\begin{minipage}[b]{\linewidth}
\centering
\centerline{\includegraphics[width = \textwidth]{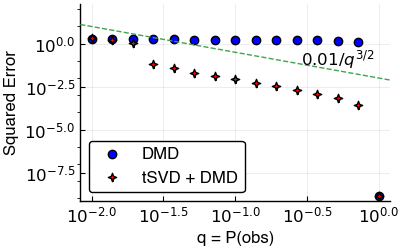}}
\subcaption{The squared estimation error of $\widehat{Q}$ as in (\ref{eq:missingloss}).}
\end{minipage}
\begin{minipage}[b]{\linewidth}
\centering
\centerline{\includegraphics[width = \textwidth]{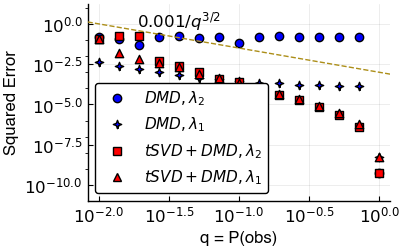}}
\subcaption{The squared estimation error of the eigenvalues $\widehat{\lambda}_i$ as in (\ref{eq:eigenvalue_missing}). }
\end{minipage}
\end{multicols}
\vspace{-0.7cm}
\caption{{Here, we verify Theorem \ref{thm:missing}. We fix the sample size $n = 10^4$, and vary the observation probability. We simulate from model (\ref{eq:start model}) with a rank $2$ cosine signal, using $\omega_1 = 0.25$ and $\omega_2 = 2$. We fix $p = 2000$ and use a non-orthogonal $Q$. We fix $d_1 = 2$ and $d_2 = 1$. We plot the error for the rank-$2$ truncated SVD (tSVD) followed by DMD, and for just DMD (both with a lag of $1$). The results show that the truncated SVD offers a tangible benefit over vanilla DMD. We also plot lines above the samples indicating that the error from the rank-$2$ tSVD + DMD algorithm is bounded by $O(1/q^{3/2})$. Note that the theoretical rate for the eigenvalue error is likely conservative.}}
\label{fig:missing_fix_n}
\end{figure*}

\begin{figure*}[htb]
\begin{multicols}{2}
\begin{minipage}[b]{\linewidth}
\centering
\centerline{\includegraphics[width = \textwidth]{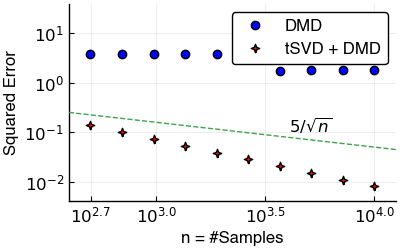}}
\subcaption{The squared estimation error of $\widehat{Q}$ as in (\ref{eq:missingloss}). }
\end{minipage}
\begin{minipage}[b]{\linewidth}
\centering
\centerline{\includegraphics[width = \textwidth]{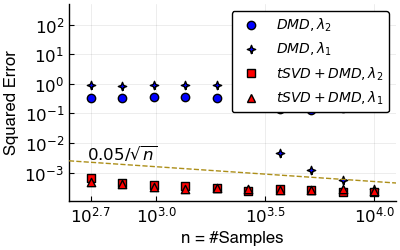}}
\subcaption{The squared estimation error of the eigenvalues $\widehat{\lambda}_i$ as in (\ref{eq:eigenvalue_missing}). }
\end{minipage}
\end{multicols}
\vspace{-0.7cm}
\caption{{Here, we verify Theorem \ref{thm:missing}. We fix the observation probability $q = 0.1$, and vary the sample size $n$. We simulate from model (\ref{eq:start model}) with a rank $2$ cosine signal, using $\omega_1 = 0.25$ and $\omega_2 = 2$. We fix $p = 2000$ and use a non-orthogonal $Q$. We fix $d_1 = 2$ and $d_2 = 1$. We plot the error for the rank-$2$ truncated SVD (tSVD) followed by DMD, and for just DMD (both with a lag of $1$). The results show that the truncated SVD offers a tangible benefit over vanilla DMD. We also plot lines above the samples indicating that the error from the rank-$2$ tSVD + DMD algorithm is bounded by $O(1 / \sqrt{n})$. Note that the theoretical rate for the eigenvector error is likely conservative.}}
\label{fig:missing_fix_q}
\end{figure*}

\subsection{Comparison with AMUSE/SOBI}

We end this section with a comparison of DMD with the AMUSE/SOBI method for source separation \citep{miettinen2016separation}. Once again we simulate from model (\ref{eq:start model}) with a rank $k = 2$ cosine signal, using $\omega_1 = 0.25$ and $\omega_2 = 2$. We fix $p = 500$, use a $Q$ with non-orthogonal columns, and $d_1 = 2$ and $d_2 = 1$. We use a lag of $1$ for the SOBI algorithm (in this case, it is the AMUSE algorithm as we use a single lag). We present these results in Figure \ref{fig:SOBI}, where we observe that DMD outperforms AMUSE. We note that with some tuning/lag selection, it is possible that SOBI may do better than DMD, but as DMD uses a single lag, SOBI/AMUSE with a single lag is perhaps a fairer comparison. Note that we perform the comparison on a deterministic signal. 

{The theoretical results for SOBI and AMUSE are asymptotic consistency statements, i.e., in the large sample limit, if the latent signals are statistically independent, we may consistently (in a statistical sense) recover them \cite{miettinen2016separation, tong1990amuse}. Other Independent Component Analysis (ICA) methods for this problem have similar statements \cite{chen2006efficient}. It is important to note that here, we have a much weaker assumption (uncorrelatedness at two lags as opposed to independence) and that our results are finite sample bounds.}

\begin{figure*}[htb]
\begin{multicols}{2}
\begin{minipage}[b]{\linewidth}
\centering
\centerline{\includegraphics[width = \textwidth]{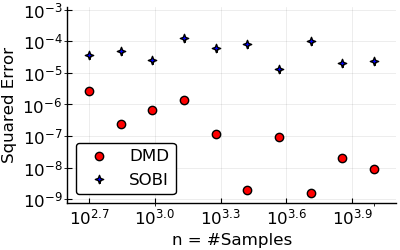}}
\subcaption{The squared estimation error of $\widehat{Q}$ as in (\ref{eq:cos_bound}).}
\end{minipage}

\begin{minipage}[b]{\linewidth}
\centering
\centerline{\includegraphics[width = \textwidth]{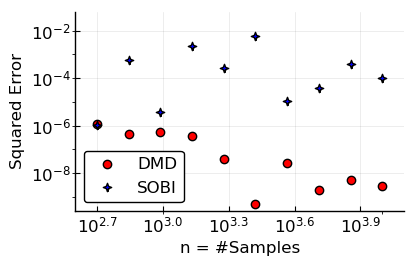}}
\subcaption{The squared estimation error of $\widehat{S}$ as in (\ref{eq:S_err}).}
\end{minipage}
\end{multicols}
\vspace{-0.7cm}
\caption{Here, we present results for DMD and AMUSE/SOBI. We simulate from model (\ref{eq:start model}) with a rank $2$ cosine signal, using $\omega_1 = 0.25$ and $\omega_2 = 2$. We fix $p = 500$ and use a $Q$ with non-orthogonal columns. We fix $d_1 = 2$ and $d_1 = 1$. We plot the estimation error of $\widehat{Q}$ and $\widehat{S}$ and compare the performance of DMD with AMUSE/SOBI for a lag of $1$. With a lag of $1$, DMD  outperforms AMUSE/SOBI.}

\label{fig:SOBI}
\end{figure*}

% Section: Real Data Algorithm %
\section{Dynamic Mode Factorization of a Time Series Data Matrix}\label{sec:realdata}

We present the Dynamic Mode Factorization (DMF) algorithm for real data in Algorithm \ref{alg:dmf}. We take the data matrix $X$ and a lag $\tau$ as inputs, and return a factorization of $X$. Our goal is to write $X = Q C^T$, where the columns of $Q$ have unit norm. If the matrix has missing entries then we fill in the missing entries with zeroes and then compute the rank $k$ (assumed known) truncated SVD approximation of the matrix as suggested by the analysis in Section \ref{sec:missing}. We assume henceforth that we are working with this filled-in matrix. If the data matrix has zero mean columns, then we estimate the column-wise mean of $X$ and subtract it to form $\overline{X}$:
\begin{equation}\label{eq:mu_Xmu}
\widehat{\bmu} = \frac{1}{n} \sum_{i = 1}^n \xx_i 
\textrm{ so that }
\overline{X} = X - \widehat{\bmu} \bones_n^T.
\end{equation}
Next, we define $\overline{X}_{(0)}^{\tau}$ and $\overline{X}_{(1)}^{\tau}$ analogously to (\ref{eq:X0X1tau}), and form
$\widehat{A}_{\tau} = \overline{X}_{(1)}^{\tau} \left[\overline{X}_{(0)}^{\tau}\right]^{+}$.
The eigenvectors of $\widehat{A}_{\tau}$ are the columns of $\widehat{Q}$, so that
$\widehat{C}^T = \widehat{Q}^{-1} \widehat{\bmu} \bones_n^T + \widehat{Q}^{-1} \overline{X}$.
Note that for a real dataset, we care about $C$ rather than $S$: the scale of our data matters, as does the mean. 

% Algorithmic %    
\renewcommand{\algorithmicrequire}{\textbf{Input:}}
\renewcommand{\algorithmicensure}{\textbf{Return:}}
\begin{algorithm*}
\begin{algorithmic}[1]
\small
\REQUIRE Data $X = \begin{bmatrix}\xx_1 &  \xx_2 &  \ldots & \xx_n\end{bmatrix}$, Integer lag $0 < \tau < n$.
\renewcommand{\algorithmicrequire}{\textbf{Goal:}}
\REQUIRE $X = \widehat{Q} \widehat{C}^T.$
\STATE Compute  $\widehat{\bmu}$ and $\overline{X} = \begin{bmatrix} \bar{\xx}_1 & \bar{\xx}_2 & \ldots & \bar{\xx}_n\end{bmatrix}$ as in (\ref{eq:mu_Xmu}).
\STATE Form $\overline{X}_{(0)}^{\tau} = \begin{bmatrix}\bar{\xx}_1 & \bar{\xx}_2 & \ldots & \bar{\xx}_{n - \tau}\end{bmatrix} \textrm{ and } \overline{X}_{(1)}^{\tau} = \begin{bmatrix}\bar{\xx}_{1 + \tau} & \bar{\xx}_{2 + \tau} & \ldots & \bar{\xx}_{n}\end{bmatrix}$.
\STATE Compute $\widehat{A}_{\tau} = \overline{X}_{(1)}^{\tau} \left[\overline{X}_{(0)}^{\tau}\right]^{+}$.
\STATE Compute $\widehat{A}_{\tau} = \widehat{Q} \widehat{\Lambda} \widehat{Q}^{-1}$ with eigenvalues  sorted by decreasing order of magnitude.
\STATE Compute $\widetilde{C}^T = \widehat{Q}^{-1} \overline{X}$.
\STATE Compute $\widehat{C}^T = \widehat{Q}^{-1} \widehat{\bmu} \bones_n^T + \widetilde{C}^T$.
\ENSURE $\widehat{Q}$, $\widehat{C}$.
\end{algorithmic}
\caption{Dynamic Mode Factorization} \label{alg:dmf}
\end{algorithm*}

\subsection{Application: Source Separation}

Next we illustrate that Algorithm \ref{alg:dmf} can unmix mixed audio signals. The first signal contains the sound of a police siren, and the second contains a music segment. The two signals have $n = 50000$ samples taken at $8$ kHz, for a duration of $6.25$ seconds each. We de-mean and scale the signals to the range $[-1, 1]$, and form an $n \times 2$ matrix $C$ with these scaled signals as columns. We mix the signals with 
$Q = \frac{1}{\sqrt{5}} \begin{bmatrix} 1 & 2 \\ 2 & 1\end{bmatrix}$,
and generate a $2 \times n$ data matrix $X = \widehat{Q} C^T$ of the mixed signals, as in (\ref{eq:our model}). Note that the $Q$ matrix does not have orthogonal columns. Figures (\ref{fig:audio}-e) and (f) show the estimates 
$\widehat{C} = \left(Q^{+} X\right)^T$
produced by the DMF algorithm with a lag of $\tau = 1$, when $X$ is the input as in Figures (\ref{fig:audio}-c) and (d). Employing PCA on $X$ does not work well here because the mixing matrix $Q$ is not orthogonal. Figures  \ref{fig:audio}-(g) and (h) show that PCA fails where the DMD algorithm succeeds. For completeness, in Figures \ref{fig:audio}-(i) and (j) we also display the results from using kurtosis-based ICA to unmix the signals. We observe that ICA performs well, but not as well as DMF (or as quickly). 

\begin{figure*}[!htb]
\begin{multicols}{6}
\begin{minipage}[b]{\linewidth}
\centering
\centerline{\includegraphics[width = \textwidth, trim = {0 0 0 0.0cm}, clip]{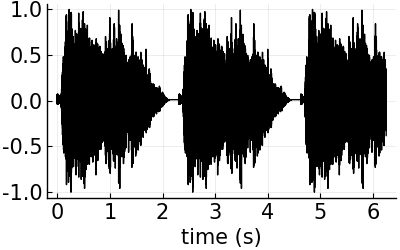}}
\subcaption{Audio 1}
\end{minipage}
\begin{minipage}[b]{\linewidth}
\centering
\centerline{\includegraphics[width = \textwidth, trim = {0 0 0 0.0cm}, clip]{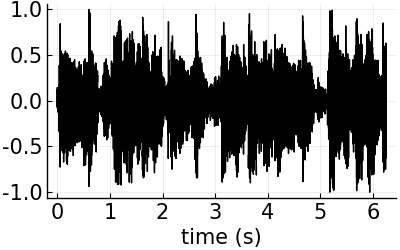}}
\subcaption{Audio 2}
\end{minipage}
\begin{minipage}[b]{\linewidth}
\centering
\centerline{\includegraphics[width = \textwidth, trim = {0 0 0 0.0cm}, clip]{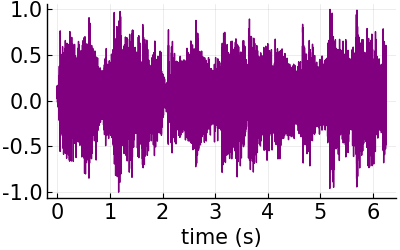}}
\subcaption{Mixed 1}
\end{minipage}
\begin{minipage}[b]{\linewidth}
\centering
\centerline{\includegraphics[width = \textwidth, trim = {0 0 0 0.0cm}, clip]{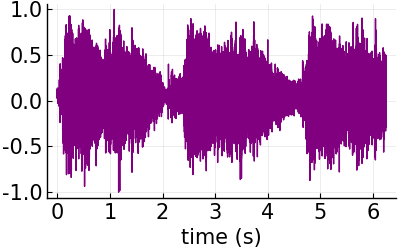}}
\subcaption{Mixed 2}
\end{minipage}
\begin{minipage}[b]{\linewidth}
\centering
\centerline{\includegraphics[width = \textwidth, trim = {0 0 0 0.0cm}, clip]{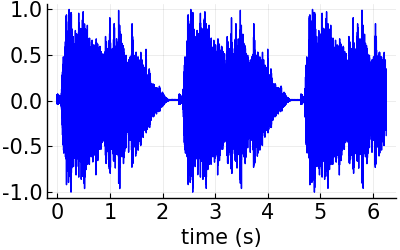}}
\subcaption{DMD 1}
\end{minipage}
\begin{minipage}[b]{\linewidth}
\centering
\centerline{\includegraphics[width = \textwidth, trim = {0 0 0 0.0cm}, clip]{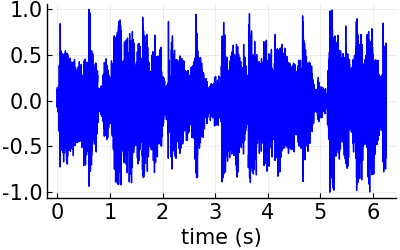}}
\subcaption{DMD 2}
\end{minipage}
\begin{minipage}[b]{\linewidth}
\centering
\centerline{\includegraphics[width = \textwidth, trim = {0 0 0 0.0cm}, clip]{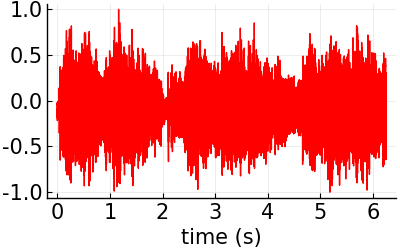}}
\subcaption{PCA 1}
\end{minipage}
\begin{minipage}[b]{\linewidth}
\centering
\centerline{\includegraphics[width = \textwidth, trim = {0 0 0 0.0cm}, clip]{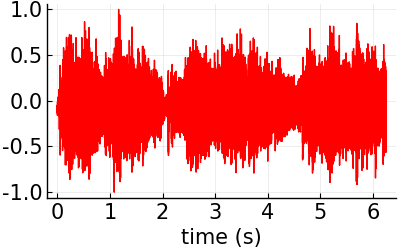}}
\subcaption{PCA 2}
\end{minipage}
\begin{minipage}[b]{\linewidth}
\centering
\centerline{\includegraphics[width = \textwidth, trim = {0 0 0 0.0cm}, clip]{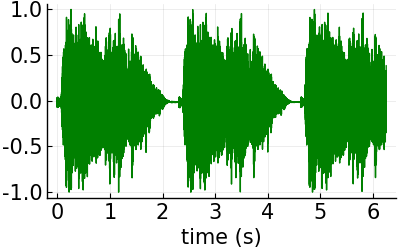}}
\subcaption{ICA 1}
\end{minipage}
\begin{minipage}[b]{\linewidth}
\centering
\centerline{\includegraphics[width = \textwidth, trim = {0 0 0 0.0cm}, clip]{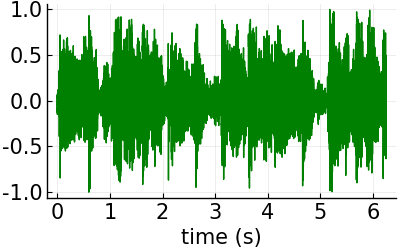}}
\subcaption{ICA 2}
\end{minipage}
\begin{minipage}[b]{\linewidth}
\centering
\centerline{\includegraphics[width = \textwidth, trim = {0 0 0 0.0cm}, clip]{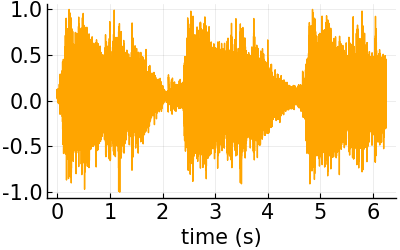}}
\subcaption{SOBI 1}
\end{minipage}
\begin{minipage}[b]{\linewidth}
\centering
\centerline{\includegraphics[width = \textwidth, trim = {0 0 0 0.0cm}, clip]{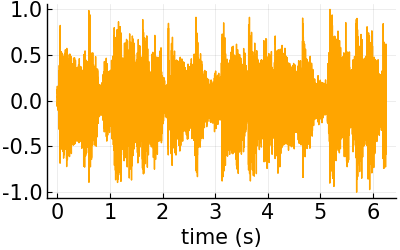}}
\subcaption{SOBI 2}
\end{minipage}
\end{multicols}
\vspace{-0.7cm}
\caption{We mix two audio signals (a police siren and a music segment), and observe that DMD successfully unmixes the signals. The squared estimation error for the unmixed signals is $2.978 \times 10^{-5}$. However, we observe that the SVD cannot unmix the signals: the squared estimation errors for the unmixed signals is $1.000$. We also display the results of ICA, which has a squared estimation errors for the unmixed signals of $0.0015$, and SOBI, which has an error of $0.00125$.}
\label{fig:audio}
\vspace{-0.4cm}
\end{figure*}

\subsection{Application: Changepoint Detection}\label{sec:changepoint}

Often, real time series contain one or more changepoints. That is, there are points in time at which the distribution or characteristics of the signal changes. In the context that we are working in, perhaps the data may exhibit a transition between modes; {we consider such an example in Figure \ref{fig:changepoint}.} In this setting, we fix $p = 4$, $k = 4$, and use 
$Q = \frac{1}{\sqrt{5}} \begin{bmatrix} 1 & 0 & 0 & 2 \\ 2 & 1 & 0 & 0 \\ 0 & 2 & 1 & 0 \\ 0 & 0 & 2 & 1 \end{bmatrix}$.
We fix $n = 1000$, and generate $C$ as follows. The first $500$ samples of $\cc_1$ are a realization of an AR(2) process with parameters $(0.2, 0.7)$, and the remaining $500$ samples are identically zero. The first $500$ samples of $\cc_2$ are identically zero, and the remaining $500$ are a realization of an AR(2) process with parameters $(0.3, 0.5)$. The first $500$ samples of $\cc_3$ are generated as $\cos 2 t$, and the remaining $500$ are identically zero.  The first $500$ samples of $\cc_4$ are identically zero, and the remaining $500$ are generated as $\cos t/2$. 

We hope that our algorithm estimates $Q$ and $S$ with low error, and that our estimated $S$ correctly captures the changepoints. That is, we hope to \emph{visually} be able to pick out when a changepoint occurs. Indeed, we find that the squared error for both $Q$ is approximately $0.069$ and that for $S$ is $0.035$, and that the estimated signals are correctly identified. Moreover, the changepoints are clearly visible. Note that PCA fails to pick out the individual signals, while preserving the changepoints; this is expected behavior, due to the non-orthogonality of the mixing. Kurtosis-based ICA also fails, as the two AR processes have Gaussian marginals. 

\begin{figure*}[!htb]
\begin{multicols}{4}
\begin{minipage}[b]{\linewidth}
\centering
\centerline{\includegraphics[width = \textwidth, trim = {0 0 0 0.5cm}, clip]{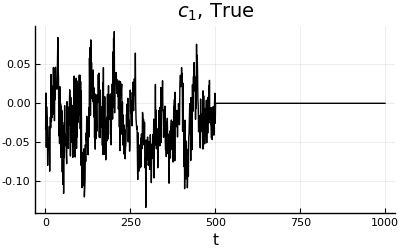}}
\subcaption{$\cc_1$}
\end{minipage}
\begin{minipage}[b]{\linewidth}
\centering
\centerline{\includegraphics[width = \textwidth, trim = {0 0 0 0.5cm}, clip]{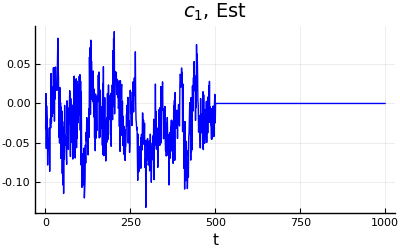}}
\subcaption{$\widehat{\cc}_1, DMD$}
\end{minipage}
\begin{minipage}[b]{\linewidth}
\centering
\centerline{\includegraphics[width = \textwidth, trim = {0 0 0 0.5cm}, clip]{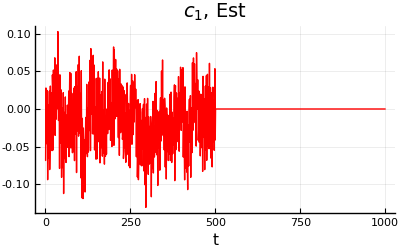}}
\subcaption{$\widehat{\cc}_1, PCA$}
\end{minipage}
\begin{minipage}[b]{\linewidth}
\centering
\centerline{\includegraphics[width = \textwidth, trim = {0 0 0 0.7cm}, clip]{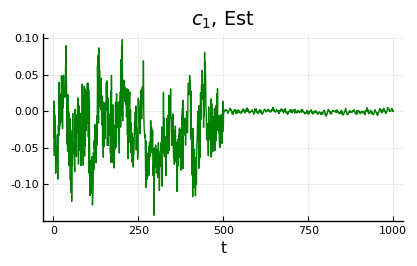}}
\subcaption{$\widehat{\cc}_1, ICA$}
\end{minipage}
\begin{minipage}[b]{\linewidth}
\centering
\centerline{\includegraphics[width = \textwidth, trim = {0 0 0 0.5cm}, clip]{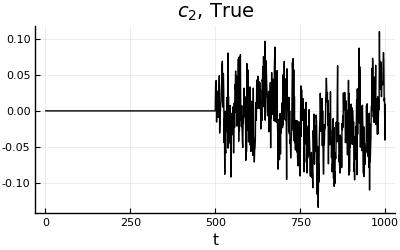}}
\subcaption{$\cc_2$}
\end{minipage}
\begin{minipage}[b]{\linewidth}
\centering
\centerline{\includegraphics[width = \textwidth, trim = {0 0 0 0.5cm}, clip]{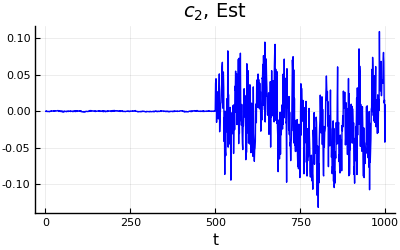}}
\subcaption{$\widehat{\cc}_2$, DMD}
\end{minipage}
\begin{minipage}[b]{\linewidth}
\centering
\centerline{\includegraphics[width = \textwidth, trim = {0 0 0 0.5cm}, clip]{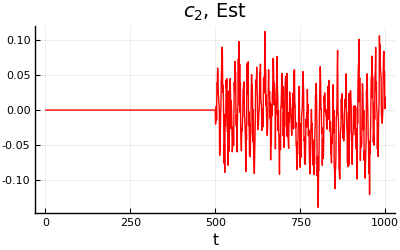}}
\subcaption{$\widehat{\cc}_2$, PCA}
\end{minipage}
\begin{minipage}[b]{\linewidth}
\centering
\centerline{\includegraphics[width = \textwidth, trim = {0 0 0 0.7cm}, clip]{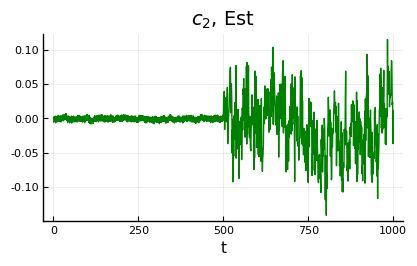}}
\subcaption{$\widehat{\cc}_2$, ICA}
\end{minipage}
\begin{minipage}[b]{\linewidth}
\centering
\centerline{\includegraphics[width = \textwidth, trim = {0 0 0 0.5cm}, clip]{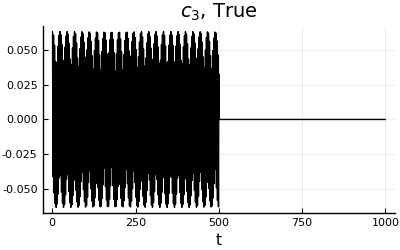}}
\subcaption{$\cc_3$}
\end{minipage}
\begin{minipage}[b]{\linewidth}
\centering
\centerline{\includegraphics[width = \textwidth, trim = {0 0 0 0.5cm}, clip]{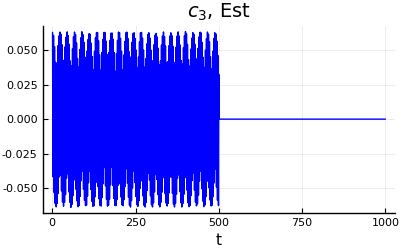}}
\subcaption{$\widehat{\cc}_3$, DMD}
\end{minipage}
\begin{minipage}[b]{\linewidth}
\centering
\centerline{\includegraphics[width = \textwidth, trim = {0 0 0 0.5cm}, clip]{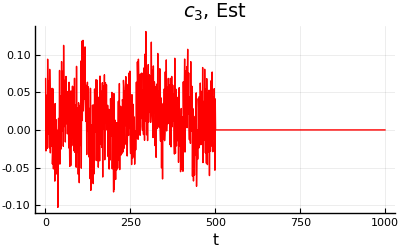}}
\subcaption{$\widehat{\cc}_3$, PCA}
\end{minipage}
\begin{minipage}[b]{\linewidth}
\centering
\centerline{\includegraphics[width = \textwidth, trim = {0 0 0 0.7cm}, clip]{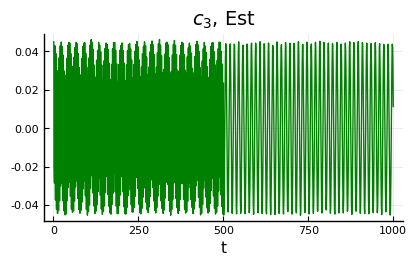}}
\subcaption{$\widehat{\cc}_3$, ICA}
\end{minipage}
\begin{minipage}[b]{\linewidth}
\centering
\centerline{\includegraphics[width = \textwidth, trim = {0 0 0 0.5cm}, clip]{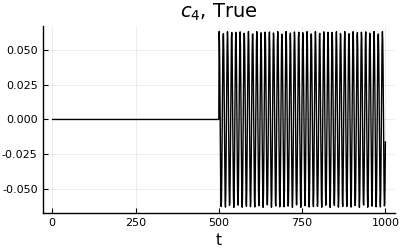}}
\subcaption{$\cc_4$}
\end{minipage}
\begin{minipage}[b]{\linewidth}
\centering
\centerline{\includegraphics[width = \textwidth, trim = {0 0 0 0.5cm}, clip]{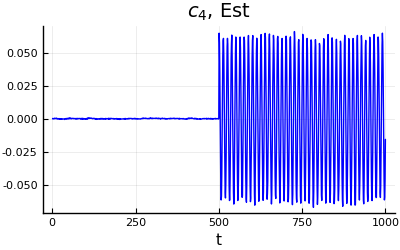}}
\subcaption{$\widehat{\cc}_4$, DMD}
\end{minipage}
\begin{minipage}[b]{\linewidth}
\centering
\centerline{\includegraphics[width = \textwidth, trim = {0 0 0 0.5cm}, clip]{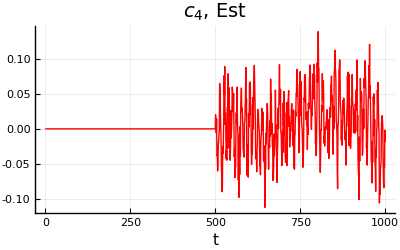}}
\subcaption{$\widehat{\cc}_4$, PCA}
\end{minipage}
\begin{minipage}[b]{\linewidth}
\centering
\centerline{\includegraphics[width = \textwidth, trim = {0 0 0 0.7cm}, clip]{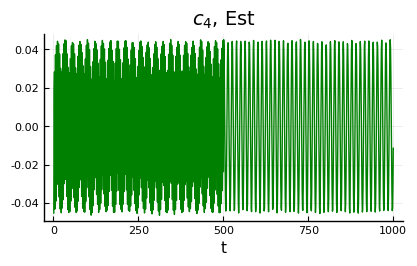}}
\subcaption{$\widehat{\cc}_4$, ICA}
\end{minipage}
\end{multicols}
\vspace{-0.7cm}
\caption{We generate $k = 4$ signals of length $n = 1000$, and mix them. Each signal has a changepoint, in that it switches from all zeros to a definite, non-zero signal. We find that the DMF algorithm perfectly captures the underlying signals, in addition to estimating $Q$ and $S$ (squared errors of $0.0098$ and $0.0096$, respectively) very well. We plot the estimated $\cc_i$ beside the true signals, and observe perfect overlap. As a comparison, we plot the results from using PCA and ICA below those from DMD. We observe that PCA fails dramatically, due to the non-orthogonality of the mixing, and that ICA does as well, due to the Gaussianity of the marginal distributions of the AR(2) processes.}
\label{fig:changepoint}
\vspace{-0.4cm}
\end{figure*}

% Section: Conclusions %
\section{Conclusions} \label{sec:conclusions}

Our analysis has revealed that DMD unmixes deterministic signals and stationary, ergodic time series that are uncorrelated at a lag of $1$ time-step. We have analyzed the unmixing performance of  DMD in the finite sample setting with and without randomly missing data, and have introduced and analyzed a natural higher-lag extension of DMD. We have provided numerical simulations to verify our theoretical results. We have shown (empirically) how the higher lag DMD can outperform conventional (lag-1) DMD for time series for which  there is a higher autocorrelation at higher lags than at lag 1: this is a natural extension of DMD that practitioners should adopt and experiment with. Moreover, we showed how DMD (like ICA-family methods) can successfully solve the cocktail party problem. Our results reveal why DMD will succeed in unmixing Gaussian time series while kurtois-based ICA fails, and also why applying DMD to a multivariate mixture of Fourier series type data, like in the eigen-walker model, can better reveal non-orthogonal mixing matrices in a way that PCA fundamentally cannot. 

There many directions for extending this research. {Analyzing and improving the performance of DMD and the tSVD-DMD algorithm and comparing it to that of SOBI in the noisy, finite sample setting is a natural next step. We have taken some preliminary steps in this direction in \cite{prasadan_tSVD_DMDconf}, where we have given performance bounds for the tSVD-DMD algorithm.} Additionally, selecting a lag at which to perform DMD is an open problem. Note that the performance of SOBI is known to be sensitive to the choice of the lag parameter \citep{tang2005recovery}, and that in Figure \ref{fig:arma}, we presented an example of a mixed time series for which $\tau$-DMD with $\tau = 2$ outperforms conventional ($\tau = 1$) DMD. One might recast the lag selection problem into a problem of optimal weight selection for a weighted multi-lag DMD setup where we consider the eigenvectors of the matrix  
$\widehat{A}_{\textrm{agg}} = \sum_{i = 1}^l w_i \widehat{A}_{\tau_i}$,
where $\widehat{A}_{\tau_i}$ is the matrix in (\ref{eq:Ahattau}) and we optimize for the weights $w_i$ which yield the best estimate for the mixing matrix $Q$ in (\ref{eq:QandS}). There are intriguing connections between this formulation and spectral density estimation in time series analysis \citep{parzen1957consistent} and multi-taper spectral estimation \citep{babadi2014review, hanssen1997multidimensional, anden2018multitaper} that suggest ways of improving the performance of DMD, and also SOBI (as the work in \citet{tichavsky2006computationally} does), in the presence of finite, noisy data in a manner that makes it robust to the lag selection misspecification. 

Finally, non-linear extensions of this work, particularly in the design and analysis of provably convergent DMD-based unmixing on non-linearly mixed ergodic time series are of interest and would complement related works on non-linear ICA  \citep{almeida2003misep, eriksson2002blind,  hyvarinen2016unsupervised, matsuda2018estimation, hyvarinen2017nonlinear, brakel2017learning, hyvarinen2018nonlinear, grais2014deep, amari1995recurrent, yang2019learning} and non-linear DMD \citep{williams2015data, tu2014dynamic}. 
 
\subsection*{Acknowledgements}

We thank Amit Surana and J. Nathan Kutz for introducing us to and intriguing us with their research applying the DMD algorithm during their respective seminars at the University of Michigan. We thank MIDAS and the MICDE seminar organizers for inviting them and especially thank J. Nathan Kutz for his thought-provoking statement that ``DMD is Fourier meets the Eigenvalue Decomposition''. This thought-provoking comment seeded our inquiry and led to the formulation in Section \ref{ssec:costhm}, from which the rest of our results flowered. We thank Hao Wu and Asad Lodhia for their detailed comments and suggestions on earlier versions of this manuscript, and particularly Florica Constantine for her suggestions and numerous edits. We thank Harish Ganesh for his perspectives and thoughts on the use of DMD in its original domain of experimental fluid mechanics. We thank Alfred Hero and Jeff Fessler for feedback and suggestions, and Shai Revzen for his thought-provoking comments and insights about when DMD does and does not work in real-world settings--they have provided us with fodder for many more questions than this work answers. The Julia package \url{https://github.com/aprasadan/DMF.jl} contains  code to reproduce all the simulations herein.

% \FloatBarrier

\appendix

% Proof of general theorem
\section{Proof of Theorem \protect\ref{thm:shift} for $\tau = 1$} \label{sec:proof_gen}

Recall the definitions of $X_{(0)}$ and $X_{(1)}$ from (\ref{eq:X0X1}). Noting that $X = Q D S^T$, we may define $S_{(0)}$ and $S_{(1)}$, where
\begin{align}\label{eq:S0S1}
S_{(0)} = \begin{bmatrix} 
s_{1, 1} & s_{2, 1} & \cdots & s_{k, 1} \\
s_{1, 2} & s_{2, 2} & \cdots & s_{k, 2} \\
\vdots & \vdots & \cdots & \vdots \\
s_{1, n - 1} & s_{2, n - 1} & \cdots & s_{k, n - 1}
\end{bmatrix}
\textrm{ and }
S_{(1)} = \begin{bmatrix} 
s_{1, 2} & s_{2, 2} & \cdots & s_{k, 2} \\
s_{1, 3} & s_{2, 3} & \cdots & s_{k, 3} \\
\vdots & \vdots & \cdots & \vdots \\
s_{1, n} & s_{2, n} & \cdots & s_{k, n}
\end{bmatrix}.
\end{align}
Then, we have that 
\begin{equation} \label{eq:X0X1_S}
X_{(0)} = Q D S_{(0)}^T \textrm{ and } X_{(1)} = Q D S_{(1)}^T.
\end{equation}

We make the key observation that 
\begin{equation} \label{eq:S1 split}
\begin{split}
S_{(1)}^T = \begin{bmatrix} \sss_{1, 2} & \sss_{1, 3} & \cdots & \sss_{1, n - 1} & \sss_{1, 1} \\  \vdots & \vdots & \cdots & \vdots & \vdots  \\ \sss_{k, 2} & \sss_{k, 3} & \cdots & \sss_{k, n - 1} & \sss_{k, 1}\end{bmatrix} + \begin{bmatrix} 0 & \cdots & 0 & \sss_{1, n} - \sss_{1, 1} \\ \vdots & \cdots & \vdots & \vdots \\ 0 & \cdots & 0 & \sss_{k, n} - \sss_{k, 1}\end{bmatrix}.
\end{split}
\end{equation}
Let $P$ be the $(n - 1) \times (n - 1)$ lag-$1$ circular shift matrix as described in the construction of the lag-$1$ inner-product matrix $L = L_{1}$ in (\ref{eq:Ltau}). A comparison of the first term on the right-hand side in the decomposition of $S^T_{(1)}$ in (\ref{eq:S1 split}) with the column partition decomposition of $S^T_{(0)}$ in (\ref{eq:S0S1}) reveals that this first term is a lag-$1$ circular shift of the matrix $S^T_{(0)}$. Consequently, we may express $S^T_{(1)}$ as  
\begin{equation} \label{eq:S_shift}
S^T_{(1)} = S^T_{(0)} P + \Delta_1,
\end{equation}
where $S^T_{(0)} P$ is the lag-$1$ circular shift of $S^T_{(0)}$ and $\Delta_1$ is the rank $1$ error matrix given by the second term in the right-hand side of (\ref{eq:S1 split}).
%\begin{equation}
%\Delta_1 =  \begin{bmatrix} 0 & \cdots & 0 & \sss_{1, n} - \sss_{1, 1} \\ \vdots & \cdots & \vdots & \vdots \\ 0 & \cdots & 0 & \sss_{k, n} - \sss_{k, 1}\end{bmatrix}.
%\end{equation}
Thus, from (\ref{eq:X0X1_S}) we have that
\begin{equation}\label{eq:X1Delta}
X_{(1)} = Q D (S^T_{(0)}P+ \Delta_1) = Q D S^T_{(0)} P + \Delta_X,
\end{equation}
where $\Delta_X = Q D \Delta_1$.
Consequently, by substituting the expression of $X_{(1)}$ from (\ref{eq:X1Delta}) and $X_{(0)}$ from (\ref{eq:X0X1_S}), we can express $\widehat{A}$ as
\begin{equation}\label{eq:Ahat split1}
\begin{split}
\widehat{A} &= X_{(1)} X_{(0)}^{+}  %\\
%    &= (Q D S^T_{(0)}P + \Delta_X) (Q D S^T_{(0)})^+ \\
%    & = Q \left[D S_{(0)}^T P \left(S^T_{(0)}\right)^+ D^+\right] Q^+ + \widehat{\Delta}_X \\
     = Q L_{D} Q^{+} + \widehat{\Delta}_X
    \end{split}
\end{equation}
where
\begin{equation}\label{eq:widehatdeltaX}
 \widehat{\Delta}_X = \Delta_X \left(S^T_{(0)}\right)^+ D^+ Q^+
\textrm{ and }
L_D = D S_{(0)}^T P \left(S^T_{(0)}\right)^+ D^+.
\end{equation}
{Let $\textrm{diag}(\cdot)$ denote the diagonal matrix determined by the main diagonal of its argument.} Then, the matrix $L_D$ can be decomposed as 
\begin{equation}\label{eq:LD diagonal}
L_D = \underbrace{\textrm{diag}(L_D)}_{=:\Lambda}+ \Delta_L.
\end{equation}
Substituting the expression of $L_D$ in (\ref{eq:LD diagonal}) into the first term on the right hand side of (\ref{eq:Ahat split1}) gives us the expression
\begin{equation} \label{eq:Aperturbation}
\widehat{A} = Q \Lambda Q^+  + \widehat{\Delta}_A,
\textrm{ where }
\widehat{\Delta}_A = Q \Delta_L Q^+ \widehat{\Delta}_X.
\end{equation}

The essence of our proof lies in bounding the size of $\widehat{\Delta}_A$. To this end, we first unpack $\widehat{\Delta}_A$. A key observation, to be substantiated in what follows, is that we may write $S_0^{+} = S_0^T + \Delta_{Sp}$,
where $\left\|\Delta_{Sp}\right\|_2$ is small (to be quantified in what follows). When we substitute this quantity into the definition of $\widehat{\Delta}_X$ in (\ref{eq:widehatdeltaX}) and expand the terms in $\widehat{\Delta}_A$, we obtain:
\begin{equation}
\begin{split}
\widehat{\Delta}_A &= Q D \Delta_{L} D^{-1} Q^{+} +  Q D S_0^T P \Delta_{Sp}^T D^{-1} Q^{+} + Q D \Delta_1 S_0 D^{-1} Q^{+} + Q D \Delta_1 \Delta_{Sp}^T D^{-1} Q^{+}.
\end{split}
\end{equation}
It is now relatively straightforward to bound the size of $\widehat{\Delta}_A$: we bound each term individually by bounding the factors therein. The most involved part of this argument comes from bounding the size of $\Delta_{Sp}$, as we will do next. Then, we will state a bound on the size of $\widehat{\Delta}_A$. Given the bound on $\widehat{\Delta}_A$, we will appeal to results from perturbation theory to bound the deviation of the eigenvectors $\widehat{\qq}_i$ of $\widehat{A}$ from $\qq_i$. 

\subsection{Bounding \texorpdfstring{$\Delta_{Sp}$}{the Pseudoinverse Perturbation}}

We now bound the size of $\Delta_{Sp}$. We proceed in three steps, separated into lemmas. Through our lemmas, we characterize the singular vectors and values of $S_0$, so that we may understand the pseudoinverse $S_0^+$. 

\begin{lemma}[The right singular vectors of $S_0$]
The right singular vectors of $S_0$ are, up to a bounded perturbation, the columns of the $k \times k$ identity matrix, $\Ii_k$, with the $j^{th}$ column denoted by $\ee_{j, k}$. 
\end{lemma}
\begin{proof}
{$S_0^T S_0$ is a $k \times k$ matrix with diagonal entries between $1 - c_1 n^{-\alpha}$ and $1$ for some small, positive constant $c_1$ ($c_1 n^{-\alpha}$ is necessarily smaller than $1$); and off-diagonal entries bounded in size by $O(f(n))$ (recall (\ref{eq:S_coherence})).} I.e., $S_0^T S_0 = \Ii_k + \Delta_V$, $\|\Delta_V\|_F^2 = O\left(k^2 f(n)^2 + k n^{-2\alpha}\right)$.
Then, the eigenvectors of $S_0^T S_0$ are the columns of the identity matrix, up to a perturbation $\Delta_V$: $\Ii_k + \Delta_V$. 

To see that $\ee_{j, k}$ is almost an eigenvector of $S_0^H S_0$: 
$\left\|\ee_{j, k} - S_0^T S_0 e_{j, k}\right\|_2^2 = O\left(k f(n)^2 + n^{-2 \alpha}\right)$.
Hence, $\|\Delta_V\|_F^2 = O\left(k^2 f(n)^2 + k n^{-2 \alpha}\right)$.
\end{proof}

Before considering the left singular vectors and singular values, we need the following fact.
\begin{lemma} \label{lem:frac}
For $a > 0$ and $a \neq 1$, there exists a constant $b(a)$ such that 
$\frac{1}{1 - a} \leq 1 + b(a) \times a$. 
Choosing $b(a) \geq \frac{1}{1 - a}$ is sufficient. 
\end{lemma}

\begin{lemma}[The left singular vectors and the singular values of $S_0$]
The left singular vectors of $S_0$ are approximately the columns of $S_0$, and the non-zero singular values are approximately $1$. 
\end{lemma}
\begin{proof}
The left singular vectors of $S_0$ are found by normalizing the columns of $S_0$ times the right singular vectors. I.e., $S_0 \left[\Ii + \Delta_V\right],$ but normalized. The size of $S_0 \Delta_V$ can be bounded by $\left\|S_0 \Delta_V\right\|_F^2 = O\left(k^3 f(n)^2 + k^2 n^{-2 \alpha}\right)$, since $\|S_0\|_F^2 \leq \|S\|_F^2 = k$. {Moreover, the norms of individual columns are bounded above by $1$ and below by 
$${\small \sqrt{1 - c_2 (k f(n)^2 + n^{-2 \alpha})} \geq 1 - c_3 \left(k^{1/2} f(n) + n^{-\alpha}\right)},$$
where $c_2$ and $c_3$ are some small, positive constants. Using Lemma (\ref{lem:frac}) and assuming that $c_3 (k^{1/2} f(n)$ $+ n^{-\alpha})$ is bounded away from $1$, e.g., by $9/10$ (which will be true for large enough $n$), a normalized column of $S_0 + S_0 \Delta_V$ has norm $1 + c_3 (k^{1/2} f(n) + n^{-\alpha})$.}
Then, writing the normalization as multiplication by a diagonal matrix, we have 
$(S_0 + S_0 \Delta_V) (\Ii + \Delta_N) = S_0 + S_0 \Delta_V + S_0 \Delta_V \Delta_N$.
The norm of $\Delta_N$ is bounded by $\|\Delta_N\|_F^2 = O(k^2 f(n)^2 + k n^{-2 \alpha})$. 
Then, the norm of $S_0$ minus the error terms is: 
$$\left\|S_0 - S_0 \Delta_V - S_0 \Delta_V \Delta_N\right\|_F^2 = O\left(k^3 f(n)^2 + k^2 n^{-2 \alpha}\right).$$
\end{proof}

Now, we may combine the previous results to bound $\Delta_{Sp}$.
\begin{lemma}[The Pseudoinverse of $S_0$]
The pseudoinverse of $S_0$ is $S_0^{+} = S_0^T + \Delta_{Sp}$, where $\|\Delta_{Sp}\|_F$ is small.
\end{lemma}
\begin{proof}
Writing the SVD of $S_0$ as 
$(S_0 + \Delta_U) (\Ii + \Delta_N) (\Ii + \Delta_V)^T$, applying Lemma \ref{lem:frac} to the individual elements of $\Ii + \Delta_N$ and noting that $\|\Delta_N'\|_F = \Theta(\|\Delta_N\|_F)$ yields that
the pseudoinverse is 
$(\Ii + \Delta_V) (\Ii + \Delta_N') (S_0 + \Delta_U)^T$.
Once again assuming that $f(n) \rightarrow 0$ and noting that $f(n) \leq 1$, 
\begin{equation}
\|\Delta_{Sp}\|_F^2 = O\left(k^3 f(n)^2 + k^2 n^{-2 \alpha}\right).
\end{equation}
\end{proof}

\subsection{Bounding the size of \texorpdfstring{$\widehat{\Delta}_A$}{the error in A}}

Now that we have computed the pseudoinverse of $S_0$, we may return to the main computation. Recall that we wrote
\begin{equation}\label{eq:deltawidehatAterms}
\begin{split}
\widehat{\Delta}_A &= Q D \Delta_{L} D^{-1} Q^{+} +  Q D S_0^T P \Delta_{Sp}^T D^{-1} Q^{+} + Q D \Delta_1 S_0 D^{-1} Q^{+} + Q D \Delta_1 \Delta_{Sp}^T D^{-1} Q^{+}.
\end{split}
\end{equation}
First, note that each factor of $Q$ and $Q^{\dagger}$ adds a factor of $k$ to the squared Frobenius norm. The pre- and post-multiplication by $D$ and $D^{-1}$ respectively adds a factor of $(d_1 / d_k)^2$. By assumption, $L = \left[S_0^T P S_0\right]$
is a $k \times k$ matrix with diagonal entries that are $\Theta(1)$ and off-diagonal entries that are bounded as $O(f(n))$, so that $\|\Delta_{L}\|_F^2 = O(k f^2(n))$.
Once again by assumption,
\begin{equation}\label{eq:delta1bound}
\|\Delta_1\|_F^2 = O(k n^{-2\alpha}),
\end{equation}
and $S_0$ and $S_0^H P$ each contribute factors of $k$ to the squared Frobenius norm. Then, we have 
\begin{equation} \label{eqn:deltaAbound}
\|\widehat{\Delta}_A\|_F^2 = O\left((d_1 / d_k)^2 k^6 \times [f(n)^2 + n^{-2 \alpha}]\right).
\end{equation}

\subsection{Eigenvectors and Eigenvalues}

We have written $\widehat{A}$ as $Q \Lambda Q^{\dagger} + \widehat{\Delta}_A$, and we know the size of $\widehat{\Delta}_A$. The next step is to compute the eigenvectors of $\widehat{A}$. Ideally, these are the columns of $Q$, notated by $\qq_j$ and estimated by $\widehat{\qq}_j$, which are stacked into $\widehat{Q}$. 

There are two basic propositions from the perturbation theory of eigenvalues and eigenvectors that we need to complete our analysis. First, we have the following proposition bounding the error in the eigenvalues as a consequence of \citet[Theorem~4.4]{demmel1997applied}:
\begin{proposition}
Let $\lambda_i$ be a simple eigenvalue of 
$A = Q \Lambda Q^{+}$,
where the columns of $Q$, denoted by $\qq_i$, are unit-norm, fixed, and linearly independent. Then, there is a eigenvalue $\widehat{\lambda}_i$ of the perturbed matrix
$\widehat{A} = A + \widehat{\Delta}_A$
such that 
$\left|\lambda_i - \widehat{\lambda}_j\right|^2 = O\left(\left\|\widehat{\Delta}_A\right\|_2^2\right)$.
\end{proposition}
\begin{proof}
From \citet[Theorem~4.4]{demmel1997applied}, we have that 
$$\widehat{\lambda}_i = \lambda_i + \frac{\yy_{i}^H \widehat{\Delta}_A\qq_{i}}{\yy_{i}^H \qq_{i}} + O\left(\left\|\widehat{\Delta}_A \right\|_2^2\right),$$
where $\qq_i$ is the corresponding unit-norm right eigenvector to $\lambda_i$, and $\yy_i$ is the corresponding unit-norm left eigenvector. Hence, 
$$\left|\widehat{\lambda}_i - \lambda_i\right| = O\left(\frac{\yy_{i}^H \widehat{\Delta}_A\qq_{i}}{\yy_{i}^H \qq_{i}}\right).$$
Noting that $\lambda_i$ is simple and that the $\qq_i$ are linearly independent, we have that $\yy_{i}^H \qq_{i}$ is fixed and non-zero (see \cite[Chapter~2]{wilkinson1965algebraic} for a discussion of this quantity), and we obtain the desired result. 
\end{proof}

Then, we have the following proposition as a consequence of \citet[Theorem~2]{meyer1988}:
\begin{proposition} \label{thm:evecperturb}
Let $\lambda_i$ be a simple eigenvalue of 
$A = Q \Lambda Q^{+}$
where the columns of $Q$, denoted by $\qq_i$, are unit-norm, fixed, and linearly independent. Let $\qq_i$ be the corresponding unit-norm right eigenvector $\qq_{i}$ to $\lambda_i$, and $\widehat{\qq}_{i}$ is the estimated eigenvector from
$\widehat{A} = A + \widehat{\Delta}_A$.
Then, we have that
$$\|\qq_{i} - p_{i} \widehat{\qq}_{i}\|_2^2 = O\left(\frac{\left\|\widehat{\Delta}_A\right\|_2^2}{\delta_{L}^2}\right),$$
where 
$p_{i} = \textrm{sign}\left(\widehat{\qq}_{i}^T {\qq}_{i}\right) \textrm{ and } \delta_L = \min_{j \neq l} |\lambda_{l} - \lambda_{j}|$.
\end{proposition}
\begin{proof}
As a consequence of \citet[Theorem~2]{meyer1988}, we may write 
$$\widehat{\qq}_{i} = \qq_{i} + \frac{\left(\lambda_i \Ii_p - A\right)^{D} \widehat{\Delta}_A \qq_{i}}{\yy_{i}^H \qq_{i}} + O\left(\left\|\widehat{\Delta}_A \right\|_2^2\right),$$
where $\yy_{i}$ is the corresponding unit-norm left eigenvector for $\lambda_i$, and $A^D$ denotes the Drazin Inverse (also called the Group Inverse) of $A = Q \Lambda Q^+$. The discussion in the proof of \citet[Corollary~4]{meyer1988} indicates that we may bound $\left(\lambda_i \Ii_p - A\right)^{D}$ in Proposition \ref{thm:evecperturb} by 
$\left\|\left(\lambda_i \Ii_p - A\right)^{D}\right\|_2 \leq {1}/{\delta_{L}}$.
Noting that $\lambda_i$ is simple and that the $\qq_i$ are linearly independent, we have that $\yy_{i}^H \qq_{i}$ is fixed and non-zero; see \citet[Chapter~2]{wilkinson1965algebraic} for a discussion of this quantity. Hence, we may bound
\begin{equation}
\begin{split}
\left\|\frac{\left(\lambda_i \Ii_p - A\right)^{D} \widehat{\Delta}_A \qq_{i}}{\yy_{i}^H \qq_{i}} + O\left(\left\|\widehat{\Delta}_A \right\|_2^2\right)\right\|_2^2 = O\left(\frac{\left\|\widehat{\Delta}_A\right\|_2^2}{\delta_{L}^2}\right).
\end{split}
\end{equation}
\end{proof}

Proposition \ref{thm:evecperturb} provides a bound on the individual eigenvector errors. Summing over the eigenvector errors, we have that 
$$\sum_{i = 1}^k \|\qq_{i} - p_{i} \widehat{\qq}_{i}\|_2^2 = O\left(k \frac{\left\|\widehat{\Delta}_A\right\|_2^2}{\delta_{L}^2}\right).$$
Noting that $\left\|\widehat{\Delta}_A\right\|_2^2 \leq \left\|\widehat{\Delta}_A\right\|_F^2$, we may substitute our bound from (\ref{eqn:deltaAbound}) to complete the proof.

% Proof of Cosine Corollary %
\section{Bridging Corollary \ref{thm:cos} and Theorem \ref{thm:shift} with $\tau = 1$}

When $C$ is a matrix of cosines, we may bridge the gap as follows. To apply Theorem \ref{thm:shift} to a matrix $C$ with columns $\cc_i$ of the form
\begin{equation}\label{eq:citform}
c_{it} = \cos \left( \omega_i t + \phi_i\right),
\end{equation}
we need to show that $L_{ii}$ does not tend to zero, that $L_{ij}$ does tend to zero for $i \neq j$, and that size of the elements of $S$ is bounded. Moreover, we need bounds on the convergence of the $L_{ij}$ and the elements of $S$. Recall that $L$ was defined in (\ref{eq:Ltau}), and is the matrix of circular inner products of the $\sss_i$, where the $\sss_i$, defined in (\ref{eq:qs}), are the normalized $\cc_i$ and form the columns of the matrix $S$. 

To tackle these three tasks, we require the following two identities governing sums of products of cosines:
\begin{equation} \label{eq:cos_prod}
\begin{split}
\sum_{t = 1}^n &\cos\left(\omega_1 t + \phi_1\right) \times \cos\left(\omega_2 t + \phi_2\right)  
= \frac{1}{2\left(\cos \omega_1 - \cos \omega_2\right)} \biggl(\cos\left(\omega_1 [n + 1] + \phi_1\right) \cos\left(\omega_2 n + \phi_2\right) \\ 
&- \cos\left(\omega_2 [n + 1] + \phi_2\right) \cos\left(\omega_1 n + \phi_1\right) 
- \cos \phi_2 \cos\left(\omega_1 + \phi_1\right) + \cos \phi_1 \cos\left(\omega_2 + \phi_2\right)\biggr),
\end{split}
\end{equation}
when $\omega_1 \neq \omega_2$, and
\begin{equation} \label{eq:cos_sqsum}
\sum_{t = 1}^n \cos^2 \left(\omega_1 t + \phi_1\right) = \frac{n}{2} + \frac{1}{2} \frac{\sin \left(\omega_1 n\right)}{\sin \omega_1} \cos \left(\omega_1 [n + 1] + 2 \phi_1\right).
\end{equation}

We first consider the simplest of the three tasks: the bound on the size of $S_{ij}$. Since the $\cc_i$ have entries of the form (\ref{eq:citform}), applying (\ref{eq:cos_sqsum}), we have that 
\begin{equation} \label{eq:c_norm}
\left\|\cc_i\right\|_2^2 = \frac{n}{2} + \frac{1}{2} \frac{\sin \left(\omega_i n\right)}{\sin \omega_i} \cos \left(\omega_i [n + 1] + 2 \phi_i\right).
\end{equation}
Note that if $\omega_i$ is not $0$ or $\pi$, (\ref{eq:c_norm}) behaves like $\Theta(n)$. If $\omega_i$ is $0$ or $\pi$, (\ref{eq:c_norm}) is equal to $n \cos^2 \phi_1$, which is also $\Theta(n)$: if $\cos^2 \phi_i = 0$ and $\omega_i = 0$ or $\pi$, $\cc_i$ is identically zero, and not part of a linearly independent set of vectors. Hence, the square of the norm of each $\cc_i$ is $\Theta(n)$, and the elements of $\cc_i$ are bounded in size by $1$. It follows that the elements of $S$ cannot be larger than $O(1 / \sqrt{n})$, or that $\alpha = 1/2$. 

Next, we consider the bound for $L_{ij}$ for $i \neq j$. Assuming that $\omega_i \neq \omega_j$, we may bound the right-hand size of  (\ref{eq:cos_prod}) by  
\begin{equation} \label{eq:cos_cross}
\frac{2}{\left|\cos \omega_i - \cos \omega_j\right|}.
\end{equation} 
But (\ref{eq:cos_prod}) is exactly the inner product of $\cc_i$ and $\cc_j$, for $i \neq j$. Since the elements of $L_{ij}$ are the inner products of the $\sss_i$ with $\sss_j$, dividing (\ref{eq:cos_cross}) by the norm of each $\cc_i$ yields a bound on the size of $L_{ij}$. Since the norm of each $\cc_i$ is $\Theta(\sqrt{n})$, the size of $L_{ij}$ is bounded by 
$$\left|L_{ij}\right| = O\left(\frac{1}{\sqrt{n}} \cdot \frac{1}{\left|\cos \omega_i - \cos\omega_j\right|}\right).$$
Taking the maximum over $i$ and $j$ yields that 
$\left|L_{ij}\right| = O\left(\frac{1}{\sqrt{n}} \cdot \frac{1}{\delta_{L}}\right)$,
where 
{\small $\delta_L = \min_{i \neq j} \left|\cos \omega_i - \cos\omega_j\right|$}.
Hence, we have that 
$f(n) = \frac{1}{\sqrt{n}} \frac{1}{\delta_L}$.
Note that $f(n)$ in the corollary contains a factor of $\delta_L$: this is the origin of the $\delta_L^4$ dependence, relative to Theorem \ref{thm:shift}, which has a $\delta_L^2$ dependence. 

Finally, we characterize the elements $L_{ii}$. The third and final identity we need is a version of (\ref{eq:cos_prod}) with $\omega_1 = \omega_2$ and $\phi_2 = \phi_1 + \omega_1$: 
\begin{equation} \label{eq:cos_lag1}
\begin{split}
\sum_{t = 1}^n \cos\left(\omega_1 t + \phi_1\right) \times \cos\left(\omega_1 [t + 1] + \phi_1\right) = \frac{n}{2} \cos \omega_1 + \frac{1}{2} \frac{\sin \left(\omega_1 n\right)}{\sin \omega_1} \cos \left(\omega_1 [n + 1] + 2 \phi_1\right).
\end{split}
\end{equation}
Unless $\omega_1$ is $\pi / 2$, $L_{ii}$ will not have limit $0$. For $\omega_1 \neq \pi / 2$, (\ref{eq:cos_lag1}) is $\Theta(n)$. Dividing by (\ref{eq:cos_sqsum}) yields that $L_{ii}$ is the ratio of two $\Theta(n)$ quantities: for large $n$, the mixed sine-cosine terms in both equations are negligible, so that $L_{ii}$ has limit $\cos \omega_i$.

Combining these steps, we obtain the result of Corollary (\ref{thm:cos}) from Theorem (\ref{thm:shift}). 

Note that more generally, we may write a version of  (\ref{eq:cos_lag1}) for larger lags $\tau$. That is, let $\omega_1 = \omega_2$, and $\phi_2 = \phi_1 + \tau \omega_1$, so that 
\begin{equation} \label{eq:cos_lagtau}
\begin{split}
\sum_{t = 1}^n \cos\left(\omega_1 t + \phi_1\right) \times \cos\left(\omega_1 [t + \tau] + \phi_1\right) = \frac{n}{2} \cos \left( \tau \omega_1 \right) + \frac{\sin \left(\omega_1 n\right)}{2 \sin \omega_1} \cos \left(\omega_1 [n + \tau + 1] + 2 \phi_1\right).
\end{split}
\end{equation}
That is, looking ahead to Theorem \ref{thm:shift}, unless $\omega_1 \tau$ is an odd multiple of $\pi / 2$, $L_{ii}(\tau)$ will not have limit $0$. Moreover, in the large $n$ limit, we would have $L_{ii}(\tau) = \cos \left(\tau \omega_1\right)$. 

% Proof of theorem with shift %

\section{The proof of Theorem \ref{thm:shift} for $\tau > 1$}

We may define 
\begin{subequations}\label{eq:S0S1tau}
\begin{equation}
S_{(0)}^{\tau} = \begin{bmatrix} 
s_{1, 1} & s_{2, 1} & \cdots & s_{k, 1} \\
s_{1, 2} & s_{2, 2} & \cdots & s_{k, 2} \\
\vdots & \vdots & \cdots & \vdots \\
s_{1, n - \tau} & s_{2, n - \tau} & \cdots & s_{k, n - \tau}
\end{bmatrix}
%\end{equation}
\textrm{ and } 
%\begin{equation}
S_{(1)}^{\tau} = \begin{bmatrix} 
s_{1, 1 + \tau} & s_{2, 1 + \tau} & \cdots & s_{k, 1 + \tau} \\
s_{1, 2 + \tau} & s_{2, 2 + \tau} & \cdots & s_{k, 2 + \tau} \\
\vdots & \vdots & \cdots & \vdots \\
s_{1, n} & s_{2, n} & \cdots & s_{k, n}
\end{bmatrix}.
\end{equation}
\end{subequations}
Then, we have that 
\begin{equation} \label{eq:X0X1_Stau}
X_{(0)}^{\tau} = Q W \left(S_{(0)}^{\tau}\right)^T \textrm{ and } X_{(1)} = Q W \left(S_{(1)}^{\tau}\right)^T.
\end{equation}
We make the key observation that 
\begin{equation} \label{eqn:deltatau}
\resizebox{\hsize}{!}{$
\left(S_{(1)}^{\tau}\right)^T = \begin{bmatrix} s_{1, 1 + \tau} &  \cdots & s_{1, n - \tau} & s_{1, 1} & \cdots & s_{1, \tau} \\  \vdots &  \cdots & \vdots & \vdots & \cdots  & \vdots \\ s_{k, 1 + \tau} &  \cdots & s_{k, n - \tau} & s_{k, 1} & \cdots & s_{k, \tau} \end{bmatrix} 
+ \begin{bmatrix} 0 & \cdots & 0 & s_{1, n - \tau + 1} - s_{1, 1} & \cdots & s_{1, n} - s_{1, \tau} \\ \vdots & \cdots & \vdots & \vdots & \cdots & \vdots \\ 0 & \cdots & 0 & s_{k, n - \tau + 1} - s_{k, 1} & \cdots & s_{k, n} - s_{k, \tau}\end{bmatrix}$,}
\end{equation} 
so that $\left(S_{(1)}^{\tau}\right)^T$ can be written as a $\tau$-times shift of $\left(S_{(0)}^{\tau}\right)^T$, plus an error term, $\Delta_{\tau}$, where $\Delta_{\tau}$ is the second term in (\ref{eqn:deltatau}). Mimicking the proof of Theorem \ref{thm:shift} for the $\tau = 1$ case and assuming that $\tau$ is sufficiently small reveals that the only change is that $\Delta_1$ is replaced with $\Delta_{\tau}$ in (\ref{eq:deltawidehatAterms}) and (\ref{eq:delta1bound}). Hence, we replace $n^{- 2\alpha}$ with $\tau  n^{-2 \alpha}$ in the final result.

% Details of Time Series Results %
\section{The Proof of Theorem \ref{thm:lagtautimeseries}}

In this section, we provide the details behind the results of Theorem \ref{thm:lagtautimeseries}. {Relative to the deterministic Theorems \ref{thm:shift}, Theorem \ref{thm:lagtautimeseries} differs only in that the quantities $L(\tau)$ and $d_i$ are random variables, where these quantities are defined in (\ref{eq:Ltau}) and (\ref{eq:D}) respectively. Hence, it is sufficient to demonstrate that $L_{\tau}$ and the $d_i$ are close to their expected values with high probability.} In what follows, we suppress the $\tau$ dependence of $L$ and other related quantities. 

\subsection{Conditions for the convergence of $L$ to $\EE L$}

We first consider the convergence of $L$. For convergence of $L$ to its expectation, we need a series of technical assumptions on the $\cc_i$. In stating these, we mimic the notation and state the conditions for Theorem 2 (equations (1) through (4)) in \citet{hong1982autocorrelation}. Essentially, at each time $t$, we have $p$ values: we have a $p$-dimensional time series. We will denote this series as $\widetilde{\cc}_t$, with
$\widetilde{\cc}_t = \begin{bmatrix} \cc_{1, t} & \cc_{2, t} & \ldots & \cc_{p, t}\end{bmatrix}^T$.
We require that each coordinate of $\widetilde{\cc}_t$ is individually an ergodic, wide-sense (covariance) stationary process with zero mean and finite variance. Formally, if $\beps_t \in \RR^p$ is the sequence of linear innovations, we are able to write  
$\widetilde{\cc}_{t} = \sum_{j = 0}^{\infty} \kappa_j \beps_{t - j}$,
where the $\kappa_j$ are $p \times p$ matrices. We require
\begin{subequations} \label{eqn:ts_conditions}
$\sum_{j = 0}^{\infty} \|\kappa_j\|_F^2 < \infty \textrm{ and } (\kappa_0)_{il} = 1$.
Moreover, if we define 
$K(z) = \sum_{j = 0}^{\infty} \kappa_j z^j$,
for $|z| < 1$, we require that the determinant of $K(z)$ is non-zero. %:
%\begin{equation}
%\det K(z) \neq 0. 
%\end{equation}
We further require that if $\mathcal{F}_{t - 1}$ is the $\sigma$-algebra generated by $\beps_s$ for $s \leq t$, 
\begin{equation}
\EE \left[\beps_t \mid \mathcal{F}_{t - 1} \right]= \bzr_p,
\EE \left[\beps_t \beps_t^T \mid \mathcal{F}_{t - 1}\right] = \Sigma_{\epsilon},
 \textrm{ and }
  \EE \left[|(\beps_t)_i|^r \mid \mathcal{F}_{t - 1}\right] \leq \infty,
\end{equation}
\end{subequations}
for $r \geq 4$. Moreover, $\Sigma_{\epsilon}$ is a fixed, deterministic $p \times p$ matrix. 

\subsection{The convergence of $L$ to $\EE L$}

Given these many conditions, what can we say? We first consider all of the entries of $L$, diagonal and off-diagonal. Recall that the elements of $L$ are (up to a scaling of $1/n$ and some neglected terms from the circularity) the auto- and cross-correlations of the $\cc_i$ at the lag $\tau$. Let $\EE L_{ij}$ be the expected value of $L_{ij}$, for all $i$ and $j$. Applying Theorem 2 of \citet{hong1982autocorrelation} (a strengthening of Theorems 1 and 2 from \citet{hannan1974uniform}), we have that 
\begin{equation}\label{eqn:corr_bound}
\begin{split}
\max_{i, j} \max_{0 \leq \tau \leq n^{\frac{r}{2 (r - 2)}}} \left|L_{ij} - \EE L_{ij}\right| = o\left(\left(\tau \log n\right)^{2 / r} \left(\log \log n\right)^{(1 + \delta) 2 / r} n^{-1/2}\right),
\end{split}
\end{equation}
almost surely, for some $r \geq 4$ and $\delta > 0$. I.e., for any reasonably small lag, as $n$ grows (and $p$ is fixed), we expect the auto- and cross-correlations to converge to their expected values, with strongly bounded deviations. Indeed, for a threshold 
$\psi = \left(\tau \log n\right)^{2 / r} \left(\log \log n\right)^{(1 + \delta) 2 / r} n^{-1/2}$,
we have that 
\begin{equation}\label{eqn:corr_prob}
\begin{split}
\mathbb{P}\left[\max_{i, j} \max_{0 \leq \tau \leq n^{\frac{r}{2 (r - 2)}}} \left|L_{ij} - \EE L_{ij}\right| 
\geq \psi\right] = O\left(\left[\log n \left(\log \log n\right)^{1 + \delta}\right]^{-1}\right).
\end{split}
\end{equation}
Hence, as $n$ increases, the $L$ matrix is close to its expected value with high probability. 

There are two more quantities of interest. First, the separation $\delta_L$: from the discussion above, it follows that the empirical value of 
$\min_{i \neq j} \left| L_{ii} -  L_{jj}\right|$
is close to
$\delta_L = \min_{i \neq j} \left|\EE L_{ii} - \EE L_{jj}\right|$
with high probability. Moreover, the lag-$0$ auto-covariance provides values of $\EE d_1^2$ and $\EE d_k^2$. It follows that the $d_i$ are within $f(n) [1 + o(1)]$ of the $\EE d_i$. 

\subsection{The desired properties of $\EE L$}

We have established that $L$ and the other quantities has the desired convergence properties. Next, we discuss what properties we want $\EE L$ to have. Assume that we are operating at a reasonable lag $\tau$ (per the conditions above). Then, we consider the lag $\tau$ autocorrelations and cross-correlations of the $\cc_i$. We want the cross-correlations to be $0$ in expectation, and the autocorrelations to be non-zero. Note that we do not demand that the $\cc_i$ be independent or uncorrelated at every lag: just at the desired lag $\tau$. In this setup, the right-hand side of (\ref{eqn:corr_bound}) provides the bounding function $f(n)$ for the Theorem, as $\EE L_{ij} = 0$ for the off-diagonal elements. 

\subsection{Special Case: ARMA}

From Theorem 3 in \citet{hong1982autocorrelation}, in the special case of a stationary ARMA process, we may strengthen these bounds. That is, if the $\cc_i$ are drawn as contiguous realizations of an ARMA process, we may replace the right-hand side of (\ref{eqn:corr_bound}) with 
$o\left(\left(\log \log n / n\right)^{1/2}\right)$,
for non-negative lags $\tau$ such that 
$\tau = O\left(\left[\log n\right]^{a}\right)$
for some $a > 0$, and with no further work reuse the same probability bound as in (\ref{eqn:corr_prob}), with $\delta = 0$.

\subsection{Obtaining the Theorem Statements}

We have computed $f(n)$ and shown that with high probability $L$ is close to $\EE L$. We have further discussed the desired properties of $\EE L$, and shown that the $d_i$ are close to $\EE d_i$ and that $\min_{i \neq j} \left|L_{ii} - L_{jj}\right| \textrm{ is close to } \min_{i \neq j} \left|\EE L_{ii} - \EE L_{jj}\right|$.
Essentially, we have computed all of the quantities that appear in Theorem \ref{thm:shift} with relevant probabilities. In Theorem \ref{thm:shift}, we replace these quantities with their expectations, and obtain the desired result. 

% Proof of S Theorem %
\section{Proof of Theorem \ref{thm:S_bound}}

\begin{proof}
Recall that the proof of Theorem \ref{thm:shift} begins by bounding the perturbation of $\widehat{A}$ from $Q \Lambda Q^{+}$, as in written in (\ref{eq:Aperturbation}). Hence, we may note that 
$\widehat{A}^T = \left(Q^{+}\right)^T \Lambda Q^T + \widehat{\Delta}_A^T$,
and note that $\widehat{\Delta}_A^T$ has the same norm as $\widehat{\Delta}_A$. Following the rest of the proof to its conclusion reveals that we may estimate the left eigenvectors of $\widehat{A}$ with the same error bound as for the right. 

Assume that our estimate of the left eigenvectors $\left(\widehat{Q^{+}}\right)^T$ has normalized columns. Then, writing 
$\left(Q^+\right)^T = \left(\widehat{Q^{+}}\right)^T + \Delta_{Q^+}^T$,
we may write 
$\left(\widehat{Q^{+}} X\right)^T = S D + X^T \Delta_{Q^+}^T$.
Let $\beps_i$ denote the $i^{th}$ column of $X^T \Delta_{Q^+}^T$, so that 
$\widehat{\sss}_i = \frac{d_i \sss_i + \beps_i}{\left\|d_i \sss_i + \beps_i\right\|_2}$.
We may write 
$$\left\|\sss_i - \widehat{\sss}_i\right\|_2 = \left\|\sss_i \left(1 - \frac{d_i}{\left\|d_i \sss_i + \beps_i\right\|_2}\right) + \beps_i \frac{1}{\left\|d_i \sss_i + \beps_i\right\|_2}\right\|,$$
where we have implicitly assumed (without loss of generality) that $\sss_i^T \widehat{\sss}_i$ is positive. By the triangle inequality, we may write
$d_i - \|\beps_i\|_2 \leq \left\|d_i \sss_i + \beps_i\right\|_2 \leq d_i + \|\beps_i\|_2$.
Then, we have that 
\begin{equation}\label{eq:s_bound_1}
\left\|\sss_i - \widehat{\sss}_i\right\|_2 \leq \max_{\pm} \left\{\left|1 - \frac{d_i}{d_i \pm \|\beps_i\|_2}\right| + \frac{\|\beps_i\|_2}{\left|d_i \pm \|\beps_i\|_2\right|}\right\},
\end{equation}
where the maximum is taken over combinations of the $\pm$ signs in both terms. 

Before proceeding, we need the following lemma:
\begin{lemma}
Let $0 < y < x$, and assume that there is a constant $c > 0$ such that $x > 1/c$. Then, 
$$\left|1 - \frac{x}{x \pm y}\right| < c y \textrm{ and } \left|\frac{y}{x \pm y}\right| < c y.$$
\end{lemma}

Continuing, if $\|\beps_i\|_2 < d_i$ for all $i = 1, 2, \ldots, k$, then by applying the lemma to each term in the right-hand side of (\ref{eq:s_bound_1}) with 
$c = {2}/{d_k}$,
we have that 
$\left\|\sss_i - \widehat{\sss}_i\right\|_2 \leq ({4}/{d_k}) \|\beps_i\|_2$.
Hence, summing over all $i$ yields that 
$$\sum_{i = 1}^k \left\|\sss_i - \widehat{\sss}_i\right\|_2^2 \leq \frac{4}{d_k^2} \sum_{i = 1}^k \|\beps_i\|_2^2 = \frac{16}{d_k^2} \|X^T \Delta_{Q^+}^T\|_F^2.$$
Recall that we have bounded $\|\Delta_{Q^+}^T\|_F^2$ by $\epsilon_{d, v}^2$, and $\|X^T\|_F^2$ by $k d_1^2$. It follows that
\begin{equation}
    \sum_{i = 1}^k \left\|\sss_i - \widehat{\sss}_i\right\|_2^2 \leq \frac{16 d_1^2}{d_k^2} k \epsilon_{d, v}^2.
\end{equation}

We have assumed that $\|\beps_i\|_2 < d_i$ for all $i = 1, 2, \ldots, k$; a sufficient condition is that 
$$\|X^T \Delta_{Q^+}^T\|_2^2 \leq \|X^T \Delta_{Q^+}^T\|_F^2 < d_k^2,$$ 
or that $k d_1^2 \epsilon_{d, v}^2 < d_k^2$.
\end{proof}

% Proof of Missing Data Theorem %
\section{Proof of Theorem \ref{thm:missing}} \label{sec:missing_proof}

Before proceeding, we remind the reader that the relevant notation and setup were presented in Section \ref{sec:missing}, and that (\ref{eq:missing_defs}) and (\ref{eq:assumptions_missing}) contain the required definitions and assumptions for the proof of the theorem. 

Following the approach taken in the proof of Theorem 2.4 in  \citet{nadakuditi2014optshrink}, we write 
\begin{equation}
\widetilde{X} = \EE_{M} \widetilde{X} + \left(\widetilde{X} - \EE \widetilde{X}\right) = q X + \left(\widetilde{X} - \EE \widetilde{X}\right) = q X + \Delta_S,
\end{equation}
where we define $\Delta_S =  \left(\widetilde{X} - \EE \widetilde{X}\right)$. We will first control the size of $\EE \left\| \Delta_S \right\|_2$. Then, noting that the tSVD-DMD algorithm performs DMD on a truncated SVD $\widehat{X}_k$ of $\widetilde{X}$, we will bound the error in estimating $X$ and $X^{+}$ from the low rank approximation of $\widetilde{X}$. That is, we will bound the deviation of the estimated singular vectors $\widehat{\uu}_i$ and $\widehat{\vv}_i$ and values $\widehat{\sigma}_i$ from the true values $\uu_i$, $\vv_i$, and $\sigma_i$, respectively, using the results from \citet{o2018random}. We will then compute the estimation error in $\left(\widehat{X}_{(0)}^{\tau}\right)^{+}$ and $\widehat{X}_{(1)}^{\tau}$, and hence write
$\widetilde{A} = \widehat{X}_{(1)}^{\tau} \left[\widehat{X}_{(0)}^{\tau}\right]^{+} = \widehat{A} + \Delta_{A}$.
We will bound the size of $\Delta_{A}$, and then bound the error in the eigenvectors of $\widetilde{A}$ from those of $\widehat{A}$. The final result will follow by an application of the triangle inequality. 

\subsection{Bounding $\EE \left\| \Delta_S \right\|_2$}

The first tool is a result of Lata\l{}a \citep{latala2005some}: 
\begin{equation}
\begin{split}
\EE \sigma_1(\Delta_S) \leq C\biggl[ \max_{i} \sqrt{\sum_{j} \EE (\Delta_S)_{i, j}^2} + \max_{j} \sqrt{\sum_{i} \EE (\Delta_S)_{i, j}^2} + \sqrt[4]{\sum_{i, j} \EE (\Delta_S)_{i, j}^4}\biggr],
\end{split}
\end{equation}
for some constant $C > 0$. 
We find that $\EE \sigma_1(\Delta_S) \leq g(n, p, k, q)$,
where 
\begin{equation}
\begin{split}
g(n, p, k, q) = O\biggl(\sqrt[4]{q (1 - q)} d_1 k \times \max\left\{ n^{1/4 - \alpha } p^{1/4 - \beta}, n^{-\alpha}, p^{-\beta}\right\}\biggr),
\end{split}
\end{equation}
Next, we need a bound on the probability that $\EE \left\| \Delta_S \right\|_2$ is close to $\left\| \Delta_S \right\|_2$. Noting that the first singular value is a $1$-Lipschitz, convex function, and that 
$\left|(\Delta_S)_{i, j}\right| = O\left(d_1 n^{-\alpha} p^{-\beta} k\right)$,
we may apply Talagrand's concentration inequality \citep[Theorem~2.1.13,~pp. 73]{tao2012topics}:
\begin{equation}
\Prob\left[\left|\sigma_1(\Delta_S) - \EE \sigma_1(\Delta_S)\right| > t \right] \leq 2 \exp\left(-c_0 t^2 \frac{n^{2\alpha} p^{2 \beta}}{d_1^2 k^2} \right) = 2 \exp\left(-c_0 \gamma t^2\right),
\end{equation}
for some constant $c_0 > 0$. 

\subsection{The Low Rank Approximation}

We apply the results from \citet{o2018random} to characterize the finite-sample performance of the low-rank approximation. Given the low-rank approximation that fills in the missing entries, we have an estimate $\widehat{X}$ of $q X$. Then, we have $\widehat{X}_0^{+}$ and $\widehat{X}_1$ that are passed into the DMD algorithm. Given $\widetilde{X}$, we will characterize how far $\widehat{X}$ is from $q X$. Then, (by assumptions on the density of $q X$) these bounds are close to those for $X_{(1)}$ and $X_{(0)}$, and we can apply them to write $\widehat{X}_{(0)}^{+}$ as $\frac{1}{q} X_{(0)}^{+} + \Delta_{S_0}$ and $\widehat{X}_{(1)}$ as $q X_{(1)} + \Delta_{S_1}$. Furthermore, we assume that we have oracular knowledge of the rank $k$. 

Before proceeding, note that we have controlled the size of the entries of $\Delta_{S}$, shown that its norm concentrates and is bounded, and bounded the expectation of the norm. Moreover, $\Delta_{S}$ is trivially zero mean and and random (from the randomness in masking the entries of $X$). Hence, we are able to apply the results from \citet{o2018random}.  

\subsubsection{The Singular Vectors of $\widetilde{X}$}

We have previously found that 
$$\Prob\left(\left|\sigma_1(\Delta_{S})\right| > \widetilde{t} \right) \leq 2 \exp\left(- c_0 \gamma (\widetilde{t} - g(n, p, k, q))^2\right).$$
Let $t = \widetilde{t} - g(n, p, k, q)$ for some $\widetilde{t}$. 

Recall that for two unit norm vectors $\xx$ and $\yy$, 
$\sin^2 \angle(\xx, \yy) = 1 - (\xx^T \yy)^2 \leq \epsilon^2$
means that if $\xx^T \yy \geq 0$, 
$$\|\xx - \yy\|_2^2 = 2\left(1 - \xx^T \yy\right) \leq 2 \left(1 - \sqrt{1 - \epsilon^2}\right) \leq 2 \epsilon^2.$$
Applying Corollary 20 from \citet{o2018random} and noting that $\|\Delta_S\|_2 \leq t$ with high probability, we have that 
\begin{equation}
\sin \angle(\vv_i, \widehat{\vv}_i) \leq 8 \sqrt{2} \frac{\sqrt{k}}{\delta_{\sigma, q}} \left[t (\sqrt{k} + 1) + t^2\right],
\end{equation}
with probability at least
\begin{equation}
\begin{split}
\left[1 - 24 \cdot 9^k \exp\left(-\gamma \frac{\delta_{\sigma, q}^2}{64}\right) - 8 \cdot 81^k \exp\left(-\gamma k \frac{t^2}{16}\right)\right] \cdot \left[1 - 2 \exp\left(- c_0 \gamma t^2 \right)\right].
\end{split}
\end{equation}

{Then, if $V$ contains the first $k$ right singular vectors of $X$, and assuming that $t \rightarrow 0$ and that $\delta_{\sigma, q} \nrightarrow 0$, we have that 
\begin{equation}
\left\|V - \widehat{V}\right\|_F = O\left(\frac{k^2 t}{\delta_{\sigma, q}}\right),
\end{equation}
with probability at least 
\begin{equation}
\begin{split}
1 - c_4 \left(9^k \exp\left(-\gamma \frac{\delta_{\sigma, q}^2}{64}\right)\right) - c_5\left(81^k \exp\left(-\gamma k \frac{t^2}{16}\right)\right) - c_6\left(\exp\left(- c_0 \gamma t^2 \right)\right),
\end{split}
\end{equation}
for some constants $c_4, c_5, c_6$. We have an identical result for $U$ and $\widehat{U}$. }

\subsubsection{The Singular Values of $\widetilde{X}$}

Applying Theorem 23 from \citet{o2018random}, we next have that 
$\widehat{\sigma}_j(\widetilde{X}) \geq \sigma_j(q X) - t$
with probability at least 
\begin{equation}
1 - 4 \cdot 9^j \exp\left(-c_0 \gamma \frac{t^2}{16}\right),
\end{equation}
and that 
\begin{equation}
\begin{split}
\widehat{\sigma}_j(\widetilde{X}) \leq \sigma_j(q X) + \sqrt{k} t + 2 \sqrt{j} \frac{t^2}{\sigma_j(q X)} + j \frac{t^3}{\left(\sigma_j (q X) \right)^2}, 
\end{split}
\label{eqn:sigma_upper}
\end{equation}
with probability at least 
\begin{equation}
\begin{split}
1 - 4 \cdot 81^k \exp\left(- c_0 \gamma \frac{t }{16}\right) - 2 \exp\left(- c_0 \gamma t^2\right).
 \end{split}
\end{equation}
It follows that 
\begin{equation}
\left|\widehat{\sigma}_i - \sigma_i\right| \leq t (\sqrt{k} + 1) + 2 \sqrt{j} \frac{t^2}{\sigma_j(q X)} + j \frac{t^3}{\left(\sigma_j (q X)\right)^2}
\end{equation}
with probability at least 
\begin{equation} \label{eq:sigmaupperprob}
\begin{split}
1 - 4 \cdot 81^k \exp\left(- c_0 \gamma \frac{t }{16}\right) - 2 \exp\left(- c_0 \gamma t^2\right) - 4 \cdot 9^j \exp\left(-c_0 \gamma \frac{t^2}{16}\right) .
 \end{split}
\end{equation}

\begin{lemma}
For positive scalars $a$, $x$, and $y$, 
$\frac{1}{x - y} \leq \frac{1}{x} + a y$
if $y > x$ or if 
$x \geq \sqrt{\frac{1}{a}} \textrm{ and } y \leq x - \frac{1}{a x}$.
Moreover, 
$\frac{1}{x + y} \geq \frac{1}{x} - a y$
if 
$x \geq \sqrt{\frac{1}{a}}, \textrm{ or if } 0 < x \leq \sqrt{\frac{1}{a}} \textrm{ and } y > \frac{1}{ax} - x$.
\end{lemma}
Applying the lemma, we find that if $t \leq \frac{3}{4} \sigma_k(q X)$ (true for sufficiently large $n$ and $p$, by assumption), we may write 
$$\left|\frac{1}{\widehat{\sigma}_j} - \frac{1}{\sigma_j(q X)}\right| \leq \frac{4}{\sigma_k^2} \left[(\sqrt{k} + 1) t + 2 \sqrt{j} \frac{t^2}{\sigma_j} + j \frac{t^3}{\sigma_j^2}\right]$$
with probability at least (\ref{eq:sigmaupperprob}).

{Then, it follows that 
\begin{subequations}
\begin{equation}
\left\|\Sigma - \widehat{\Sigma}\right\|_F = O\left(k t\right)
\textrm{ and }
\left\|\Sigma^+ - \widehat{\Sigma^+}\right\|_F = O\left(\frac{k t}{\sigma_k^2}\right),
\end{equation}
with probability at least 
\begin{equation}
1 - c_7 \left(81^k \cdot k \cdot \exp\left(- c_0 \frac{\gamma t^2}{16}\right)\right),
\end{equation}
\end{subequations}
where $c_7$ is some positive constant. We have assumed that $\sigma_k \nrightarrow 0$ and that $t^2 \gamma \nrightarrow 0$.}

\subsubsection{The Error in $\widehat{X}$}

Finally, we may combine all of the above results and write the following where if $q X = U \Sigma V^T$ is the (thin) SVD of $q X$, $\widehat{X} = \left(U + \Delta_U\right) \left(\Sigma + \Delta_{\Sigma, q} \right) \left(V + \Delta_V\right)^T$. We may then write $\widehat{X} = qX + \Delta_{X}$, where 
\begin{equation} \label{eq:missingdeltaXfull}
 \Delta_{X} = U \Sigma \Delta_V^T + U \Delta_{\Sigma, q} V^T + U \Delta_{\Sigma, q} \Delta_V^T  
+ \Delta_U \Sigma V^T + \Delta_U \Sigma \Delta_V^T + \Delta_U \Delta_{\Sigma, q} V^T + \Delta_U \Delta_{\Sigma, q} \Delta_V^T.
\end{equation}
Then we may write $\widehat{X} = q X + \Delta_{X}$, where $\Delta_{X}$ is defined as all but the first term in (\ref{eq:missingdeltaXfull}). We now plug in our bounds for the sizes of the $\Delta$ terms, note that each $U$ and $V$ add factors of $\sqrt{k}$ to the Frobenius norm, and note that $\Sigma$ adds a factor bounded by $\sqrt{k} \sigma_1(q X)$. {Then, when $g$ is sufficiently small, we have that
\begin{equation} \label{eq:missingdeltaX}
\left\|\Delta_X\right\|_F = O\left(k^3 t \frac{\sigma_1(q X)}{\delta_{\sigma, q}}\right),
\end{equation}
with probability at least 
\begin{equation}
\begin{split}
1 - c_7\left(k \cdot 81^k \exp\left(-c_0 \gamma k t^2 / 16\right)\right) - c_8\left(k \cdot 9^k \exp\left(-c_0 \gamma k \delta_{\sigma, q} / 64\right)\right),
\end{split}
\end{equation}
where $c_8$ is another positive constant.} The result for $\widehat{X}^{+}$ is similar: we may expand $\widehat{X}^{+}$ as we did for $\widehat{X}$ in (\ref{eq:missingdeltaXfull}), and obtain that with the same probability, we have $\widehat{X}^{+} = X^{+} + \Delta_{X^{+}}$, where
\begin{equation}\label{eq:missingdeltaXpinv}
\left\|\Delta_{X^{+}}\right\|_F = O\left(k^3 t \frac{1}{\delta_{\sigma, q}}\right).
\end{equation}

\subsection{Using $\widehat{X}_k$ to estimate $\widehat{A}$}

% Square of Norm of tau columns is bounded by k d_1^2 \tau n^{-2 \alpha} %
Next, we consider the estimation of $\widehat{A}$ with 
$\widetilde{A} = \widehat{X}_{(1)}^{\tau} \left[\widehat{X}_{(0)}^{\tau}\right]^{+}$. 
That is, we estimate $\widehat{X}$, and take the sub-matrices $\widehat{X}_{(1)}^{\tau}$ and $\widehat{X}_{(0)}^{\tau}$ as inputs to DMD. Our previous bounds may be applied with $g(n, p, k, q)$ replaced with $\sqrt{\tau} g(n, p, k, q)$:
note that the sum of squares of the norms of $\tau$ columns of $X$ is bounded by $k d_1^2 \tau n^{-2 \alpha}$, and all of these factors except $\tau$ appear in $g(n, p, k, q)^2$. Writing $\widehat{X}_{(1)}^{\tau} = X_{(1)} + \Delta_{X_1}$ and $\left(\widehat{X}_{(0)}^{\tau}\right)^{+} = X_{(0)}^{+} + \Delta_{X_0^{+}}$, we may write $\widetilde{A} = \widehat{A} + \Delta_{A}$, where $\Delta_A$ is the sum of all but the first term in
\begin{equation}
\widetilde{A} = X_{(1)} X_{(0)}^{+} + \Delta_{X_1} X_{(0)}^{+}  + X_{(1)} \Delta_{X_0^{+}} + \Delta_{X_1} \Delta_{X_0^{+}}.
\end{equation}
Note that we have dropped the $\tau$ dependence for ease of reading. Each factor of $X_{(1)}$ adds $\sqrt{k} \times \sigma_1(q X_{(1)})$ to the Frobenius norm, and each factor of $X_{(0)}^{+}$ adds $\sqrt{k} / \sigma_k(q X_{(0)})$. Hence, we may write
\begin{equation} \label{eq:deltaA_part1}
\left\|\Delta_{A}\right\|_F = O\left(\frac{\sqrt{k}}{\sigma_k(q X_0)} \left\|\Delta_{X_1}\right\|_F + \sqrt{k} \sigma_1(q X_1) \left\|\Delta_{X_0^{+}}\right\|_F\right).
\end{equation}

Ideally, we would have (\ref{eq:deltaA_part1}) in terms of $X$. First, note that by the Cauchy Interlacing Theorem \citep{fisk2005very}, $\sigma_1(q X_{(1)}) \leq \sigma_1(q X)$. It follows that we may replace $X_{(1)}$ with $X$ without any further work. 

Since $X_{(0)}$ has the same singular values as a version of $X$ with the last $\tau$ columns set to $0$, we may replace $X_{(0)}$ with a perturbation of $X$, denoted by $\widetilde{X}_{(0)}$: $\widetilde{X}_{(0)} = X + \widetilde{\Delta}_{X_0}$, where 
$$\left\|\widetilde{\Delta}_{X_0}\right\|_F \leq \sqrt{k \tau} d_1 n^{-\alpha} \leq \sqrt{\frac{\tau}{\sqrt{q (1 - q)}}} \times g(n, p, k, q).$$
An application of the Weyl Inequality \citep[Theorem~4.3.1]{weyl1912asymptotische} yields that 
$$\frac{1}{\sigma_k(q X_{(0)})} = \frac{1}{\sigma_k(q \widetilde{X}_{(0)})}   \leq \frac{1}{\sigma_k(q X) - q \widetilde{\Delta}_{X_0}}.$$
By assumption, $\sigma_k(X)$ does not have limit $0$. Moreover, by assumption, the norm of $\widetilde{\Delta}_{X_0}$ does have limit zero. Hence, for sufficiently large $n$, we may write 
$$\frac{1}{\sigma_k(q X_{(0)})} \leq \frac{1}{\sigma_k(q X)} + \frac{1}{q} O\left(\left\|\widetilde{\Delta}_{X_0}\right\|_F\right) \leq \frac{1}{\sigma_k(q X)} + \frac{\sqrt{\tau}}{q} O(g(n, p, k, q)).$$

Now, let $t = a g(n, p, k, q)$ for some $a > 1$. Putting the previous work together, we find that 
\begin{equation}
\left\|\Delta_{A}\right\|_F = O\left(k^{7/2} a \sqrt{\tau} g(n, p, k, q) \frac{\sigma_1(q X)}{\delta_{\sigma, q}}\right).
\end{equation}
{This bound holds with probability at least 
\begin{equation}\label{eqn:A_prob}
\begin{split}
1 - c_7\left(k \cdot 81^k \exp\left(-\left(1 - \frac{1}{a}\right)^2 c_0 \gamma \frac{\left(\sqrt{\tau} g(n, p, k, q)\right)^2}{16}\right)\right) - c_8\left(k \cdot 9^k \exp\left(-c_0 \gamma \frac{\delta_{\sigma, q}}{64}\right)\right),
\end{split}
\end{equation}
for some constants $c_7$ and $c_8$.}

\subsection{The DMD Eigenvectors}

Finally, we have previously bounded the deviation of $\widehat{A} = X_{(1)} X_{(0)}^+$ from $Q \Lambda Q^+$. We have just bounded the deviation of $\widetilde{A}$ from $\widehat{A}$ due to missing data. We may combine the effects of missing data and the deterministic noiseless deviation bound via the triangle inequality. Then, we apply the the union bound over the $k$ eigenvectors. Let $\epsilon_{d, v}^2$ be the deterministic deviation of the $\qq_k$, i.e., the right-hand side of (\ref{eq:generalboundshift}). {Then, with probability at least
\begin{equation} \label{eq:missingprobproof}
\begin{split}
1 - c_7 \left(k^2 \cdot 81^k \exp\left(-\left(1 - \frac{1}{a}\right)^2 c_0 \gamma \frac{\tau \left(g(n, p, k, q)\right)^2}{16}\right)\right) - c_8\left(k^2 \cdot 9^k \exp\left(-c_0 \gamma \frac{\delta_{\sigma, q}}{64}\right)\right),
\end{split}
\end{equation}
\begin{equation}
\begin{split}
\sum_{i = 1}^k \left\|\widehat{\qq}_i - p_i \qq_i\right\|_2^2 = O\left(\frac{\tau}{q^2} a^2 \left(g(n, p, k, q) \right)^2\frac{\sigma_1^2(X)}{\delta_{\sigma}^2} \frac{k^8}{\delta_{L}^2} + \epsilon_{d, v}^2 \right),
\end{split}
\end{equation}
where we have adapted the final step in the proof of Theorem \ref{thm:shift}. }

Finally, let $\epsilon_{d, e}^2$ be the deterministic deviation of the $L_{ii}$, i.e., the right-hand side of (\ref{eq:evalboundshift}). Once again adapting the final step in the proof of Theorem \ref{thm:shift}, we have that for each $L_{ii}$, there is an eigenvalue of $\widetilde{A}$ such that 
\begin{equation}
\left|L_{ii} - \lambda_i\right|^2 = O\left(\frac{\tau}{q^2} a^2 \left(g(n, p, k, q) \right)^2\frac{\sigma_1^2(X)}{\delta_{\sigma}^2} {k^7} + \epsilon_{d, e}^2 \right),
\end{equation}
with probability at least (\ref{eq:missingprobproof}).

\bibliographystyle{plain}
\bibliography{DMD}

\end{document}